\documentclass[10pt, a4paper, reqno, oneside]{amsart}

\usepackage{amsmath, amsopn, amsfonts, amsthm, amssymb, amscd, enumerate, multicol, scalefnt}

\usepackage{amsmath, amsfonts, amssymb, amscd, enumerate, scalefnt}

\usepackage{etex}
\usepackage{hyperref}
\usepackage{t1enc}
\usepackage{graphicx}
\usepackage[all]{xy}
\usepackage{color}
\usepackage{slashed}
\usepackage[mathscr]{euscript}

\usepackage[dvipsnames]{xcolor}
\usepackage{amscd}
\usepackage{color}
\usepackage{amssymb}
\usepackage{amsmath}
\usepackage{amsfonts}
\usepackage{latexsym}
\usepackage{verbatim}
\newcommand\bcdot{\ensuremath{
  \mathchoice
   {\mskip\thinmuskip\lower0.2ex\hbox{\scalebox{1.6}{$\cdot$}}\mskip\thinmuskip}}
   {\mskip\thinmuskip\lower0.2ex\hbox{\scalebox{1.6}{$\cdot$}}\mskip\thinmuskip}
   {\lower0.3ex\hbox{\scalebox{1.2}{$\cdot$}}}
   {\lower0.3ex\hbox{\scalebox{1.2}{$\cdot$}}}
}

\theoremstyle{plain}
\newtheorem{theo}{Theorem}[section]

\newtheorem{prop}[theo]{Proposition}

\newtheorem{cor}[theo]{Corollary}

\theoremstyle{definition}
\newtheorem{rem}[theo]{Remark}
\newtheorem{example}[theo]{Example}
\newtheorem{definition}[theo]{Definition}

\theoremstyle{plain}

\newtheorem{theorem}[theo]{Theorem}

\theoremstyle{definition}

\newtheorem{remark}[theo]{Remark}

\theoremstyle{plain}
\newtheorem{thmint}{Theorem}

\newcommand{\Aac}{\`A}

\newcommand{\g}{\mathfrak{g}}

\renewcommand{\l}{\mathfrak{l}}
\newcommand\C{{\mathbb C}}
\newcommand\R{{\mathbb R}}
\newcommand\Z{{\mathbb Z}}

\newcommand{\wt}[1]{\widetilde{ #1}}

\newcommand{\GHa}{\xrightarrow{\text{\rm GH}}}

\title[HCF on complex locally homogeneous surfaces]{Hermitian Curvature flow on complex locally homogeneous surfaces}
\author{Francesco Pediconi and Mattia Pujia}


\thanks{This work was supported by G.N.S.A.G.A. of I.N.d.A.M. The first named author was supported by project PRIN 2017 "Real and Complex Manifolds: Topology, Geometry and holomorphic dynamics" (code 2017JZ2SW5). }

\begin{document}

\begin{abstract}
We study the Hermitian curvature flow of locally homogeneous non-K\"ahler metrics on compact complex surfaces. In particular, we characterize the long-time behavior of the solutions to the flow. We also provide the first example of a compact complex non-K\"ahler manifold admitting a finite time singularity for the Hermitian curvature flow. Finally, we compute the Gromov-Hausdorff limit of immortal solutions after a suitable normalization. Our results follow by a case-by-case analysis of the flow on each complex model geometry.
\end{abstract}

\maketitle

%\tableofcontents

\section{Introduction} 

The {\em Hermitian curvature flow} (HCF shortly) is a strictly parabolic flow of Hermitian metrics introduced by Streets and Tian in \cite{ST2}. The flow evolves an initial Hermitian metric in the direction of its second Chern-Ricci curvature tensor modified with some first order terms in the torsion.

More precisely, let $(X,g_0)$ be a Hermitian manifold. The solution to the HCF starting at $g_0$ is the family of Hermitian metric $g(t)$ satisfying
\begin{equation}\label{HCF0}
\partial_t\, g(t)=-S(g(t))+Q(g(t))\,\,,\qquad g(0)=g_0\,\,,
\end{equation}
where $S(g)$ is the second Chern-Ricci curvature tensor and $Q(g)$ is a $(1,1)$-symmetric tensor which is quadratic in the torsion components of the Chern connection (see Section \ref{prel}). The flow arises from the unique Hilbert-type functional of Hermitian metrics for which $S$ is the leading term of the respective Euler-Lagrange equation. Moreover, when the starting metric is K\"ahler the HCF reduces to the K\"ahler-Ricci flow and, in the compact case, it is stable near K\"ahler-Einstein metrics with non-positive Ricci curvature \cite{ST2}.\medskip

%{\color{blue}We stress that different choices of the tensor $Q$ in \eqref{HCF0} give rise to a family of geometric flows. In particular one can choose $Q$ to preserve different geometric conditions, making each of these new flows well-suited to investigate a certain problem. Among these, one of the most studied is the {\em pluriclosed flow} (PCF shortly), which preserves the pluriclosed condition $\partial \bar \partial \omega=0$ \cite{S, S2, S3, S-solitons, S-conjecture, ST, ST3, ST4}. }

Motivated by the above arguments, we carry out an analysis of the HCF on locally homogeneous compact complex surfaces. Actually, one of the main reasons in studying this new flow is to refine the Enriques-Kodaira classification of compact complex surfaces, as canonical metrics could appear as limit points of the flow (see e.g. \cite{S-conjecture}). \smallskip

Our first main result completely characterizes the long-time behavior of locally homogeneous non-K\"ahler solutions, namely

\begin{thmint}\label{theor_A} Let $X$ be a compact complex surface and $g_0$ a locally homogeneous non-K\"ahler metric on $X$. If the solution to the HCF starting from $g_0$ develops a finite time singularity, then $X$ is a Hopf surface. Conversely, any locally homogeneous solution to the HCF on a Hopf surface collapses in finite time.
\end{thmint}

%{\color{red}\begin{thmint}\label{theor_A} Let $X$ be a compact complex surface and $g_0$ a locally homogeneous non-K\"ahler metric on $X$. The HCF starting from $g_0$ develops a finite time singularity if and only if $X$ is a Hopf surface. Furthermore, any locally homogeneous solution to the HCF on a Hopf surface collapses in finite time.
%\end{thmint}}

It is worth noting that Theorem \ref{theor_A} provides the first example of a compact manifold developing a finite-time singularity for the HCF. Moreover, we restricted our analysis to starting non-K\"ahler metrics since the behavior of K\"ahler solutions is already known (see e.g. \cite{CT,Song-T, TZ}). \smallskip

Our second main result concerns the Gromov-Hausdorff limits of immortal normalized solutions to the HCF, namely

\begin{thmint}\label{theor_B} Let $X$ be a compact complex surface, $g_0$ a locally homogeneous non-K\"ahler metric on $X$ and $g(t)$ the solution to the HCF starting from $g_0$.
\begin{enumerate}[i)]
\item\label{i_B} If $X$ is either a hyperelliptic or Kodaira surface, then $\big(X,(1{+}t)^{-1}g(t)\big)$ converges to a point in the Gromov-Hausdorff topology as $t \to \infty$.
\item\label{ii_B} If $X$ is a non-K\"ahler properly elliptic surface, then $\big(X,(1{+}t)^{-1}g(t)\big)$ converges to its base curve $(C,g_{{}_{\rm KE}})$ in Gromov-Hausdorff topology as $t \to \infty$, where ${\rm Ric}(g_{{}_{\rm KE}})=-g_{{}_{\rm KE}}$.
\item\label{iii_B} If $X$ is an Inoue surface, then $\big(X,(1{+}t)^{-1}g(t)\big)$ converges to a circle in Gromov-Hausdorff topology as $t \to \infty$.
\end{enumerate}
\end{thmint}

We point out that the arguments used to prove \eqref{ii_B} and \eqref{iii_B} in  Theorem \ref{theor_B} are analogous to those used by Tosatti and Weinkove in \cite{TW} for the {\it Chern-Ricci flow} (see also \cite{FTWZ, TW0, TWY}), and the limit spaces arising in our context are the same. We highlight that  cohomological aspects of compact complex surfaces along the Chern-Ricci flow were investigated in \cite{Ang}, and it would be interesting to carry out a similar analysis also for the HCF. \smallskip

Our results can be thought of as a first step in the study of the HCF on complex non-K\"ahler surfaces. In the same spirit of \cite{Bol} and \cite{Lott}, we expect the blowdown of any immortal locally homogeneous solution to converge to an expanding soliton. Nonetheless, at the moment we are not able to confirm this statement. In this direction, in \cite{LPV} the second named author, Lafuente and Vezzoni proved that long-time existence of left-invariant solutions to the HCF is always guaranteed on complex unimodular Lie groups and such solutions always converge under a suitable normalization to an expanding algebraic soliton (see also \cite{P-DGA}). \medskip

We stress that different choices of the tensor $Q$ in \eqref{HCF0} give rise to an interesting family of geometric flows generalizing the Ricci flow to the Hermitian non K\"ahler setting. In particular one can choose $Q$ to preserve different geometric conditions, making each of these new flows well-suited to investigate a certain problem. Among these, one of the most studied is the {\em pluriclosed flow} (PCF shortly), which preserves the pluriclosed condition $\partial \bar \partial \omega=0$ (see e.g. \cite{S3, S-solitons, S-conjecture, ST, ST4}). In \cite{Bol} Boling studied the pluriclosed flow on complex surfaces proving long-time existence and convergences results; while, In \cite{AL} Arroyo and Lafuente showed that normalized left-invariant solutions to the PCF on $2$-step nilmanifolds and almost-abelian Lie groups always converge to expanding solitons (see also \cite{EFV2,PV}). %\smallskip

Let us also mention that recently Ustinovskiy \cite{yuri, yuri2} found a new flow in the HCF family which preserves both the Griffiths-positivity and a finite dimensional space of distinguished metrics called {\it induced metrics}. We mention that related works have been recently appeared (see e.g. \cite{PodPan, P20}) and it would be interesting to analyze Ustinovkiy's flow on compact complex surfaces in the same fashion as we did in this paper. \medskip

The paper is organized as follows. In Section \ref{prel} we recall some basics on HCF, complex model geometries and Gromov-Hausdorff convergence. In Section \ref{model_sub} we explicitly compute the HCF tensor of each compact complex surface considered throughout the paper. In Section \ref{mainsec} we prove Theorem \ref{theor_A} and Theorem \ref{theor_B} by a case-by-case analysis of the involved equations. Finally, in the Appendix we explicitly write the components of the HCF tensors $K$ for our class of surfaces. \bigskip

\noindent{\it Acknowledgement.} We warmly thank Daniele Angella, Alberto Raffero and Luigi Vezzoni for their interest and helpful comments. We also thank the anonymous referee for his/her useful suggestions, which improved the presentation of the paper.

\medskip

\section{Preliminaries} \label{prel}

\subsection{Basics on HCF} \hfill \par

In the sequel, we describe the evolution equation of the Hermitian Curvature Flow on a complex manifold $X=(M,J)$. Given a Hermitian metric $g$ on $X$, we denote by $\nabla$ its Chern connection, by $\Omega$ its Chern curvature tensor $\Omega(X,Y):=[\nabla_X,\nabla_Y]-\nabla_{[X,Y]}$ and by $S$ its second Chern-Ricci curvature (we use the same convention adopted in \cite{ST2}), i.e.
$$
S_{i\bar j} := g^{\bar{\ell}k}\Omega_{k\bar{\ell}i\bar{j}} \,\, .
$$
Let also $T$ be the torsion of $\nabla$ and consider the tensor $Q=Q_{i\bar{j}}$ defined by
\begin{equation}
Q := \frac12Q^1 - \frac14Q^2- \frac12Q^3 + Q^4 \,\, , \label{defQ}
\end{equation}
where
\begin{equation} \begin{aligned}
&Q^1_{i\bar{j}}:=g^{\bar{\ell}k}g^{\bar{q}m}T_{ik\bar{q}}T_{\bar{j}\bar{\ell}m} \,\, , \quad Q^2_{i\bar{j}}:=g^{\bar{\ell}k}g^{\bar{q}m}T_{km\bar{j}}T_{\bar{\ell}\bar{q}i} \,\, , \\
&Q^3_{i\bar{j}}:=g^{\bar{\ell}k}g^{\bar{q}m}T_{ik\bar{\ell}}T_{\bar{j}\bar{q}m} \,\, , \quad Q^4_{i\bar{j}}:=\frac12g^{\bar{\ell}k}g^{\bar{q}m}(T_{mk\bar{\ell}}T_{\bar{q}\bar{j}i} +T_{\bar{q}\bar{\ell}k}T_{mi\bar{j}}) \,\, .
\end{aligned} \label{defQi} \end{equation}
Notice that in the formulas above
$$
T_{ij\bar k}:=g_{\ell\bar{k}}T_{ij}^{\ell} \,\, , \quad T_{\bar{i}\bar{j}k}:=g_{k\bar{\ell}}T_{\bar{i}\bar{j}}^{\bar{\ell}} \,\, , \quad (g^{i\bar{j}}):=(g_{i\bar{j}})^{-1} \,\, .
$$ 
Then, given a Hermitian metric $g_0$ on $X$, the evolution equation of the HCF on $X$ starting from $g_0$ is given by
\begin{equation}\label{HCF}
\partial_tg(t)=-K(g(t))\,,\qquad g(0)=g_0\,,
\end{equation}
where $K:=S-Q$. Henceforth, we will refer to $K$ as the {\it HCF tensor}.

\subsection{HCF tensor on Lie groups} \hfill \par

Let $(G,J,g)$ be a real Lie group $G$ equipped with a left-invariant Riemannian metric $g$ and a left-invariant complex structure $J$ such that $g(J\cdot\,,J\cdot\,)=g(\cdot\,,\cdot\,)$. Let also $\g:={\rm Lie}(G)$ and 
$$
\mu \in \Lambda^2\g^* \otimes \g \,\, , \quad \mu(X,Y):=[X,Y] \,\, .
$$ 
In the following, we compute the components of the HCF tensor in terms of the structure constants of $\g$. 

Let $\{Z_1,\ldots,Z_n\}$ be a left-invariant frame of $T^{1,0}G$. Since the Chern connection is the unique Hermitian connection with vanishing (1,1)-part of the torsion, it follows that
$$
\nabla_{\bar Z_{k}}Z_{\ell}=\nabla_{Z_\ell} \bar{Z}_k+\mu(\bar{Z}_k,Z_{\ell})
$$
or, in terms of the Christoffel symbols of $\nabla$
$$
\Gamma_{\bar k\ell}^r=\mu_{\bar k\ell}^r\,\,,\quad \Gamma_{k\bar \ell}^{\bar r}=\mu_{ k \bar \ell}^{\bar r}\,\,.
$$
On the other hand $\nabla J=\nabla g=0$ implies
$$
g(\nabla_{Z_k}Z_i,\bar{Z}_j)=-g(Z_i,\nabla_{Z_k}\bar{Z}_j)=-g(Z_i,\mu(Z_k,\bar{Z}_j))
$$
and hence
\begin{equation}
\Gamma_{ki}^{s}=-g^{\bar{j}s}g_{i\bar{p}}\,\mu_{k\bar{j}}^{\bar{p}} \,\, .
\label{Gkis} \end{equation}
By definition, we have 
$$
\Omega_{k\bar{\ell}i\bar{j}}=g(\nabla_{Z_k}\nabla_{\bar Z_{\ell}}Z_i,\bar{Z}_j)-g(\nabla_{\bar Z_{\ell}}\nabla_{Z_k}Z_i,\bar{Z}_j)-g(\nabla_{\mu(Z_k,\bar{Z}_{\ell})}Z_i,\bar{Z}_j) \,\, ,
$$
with 
$$
\begin{gathered}
g(\nabla_{Z_k}\nabla_{\bar Z_{\ell}}Z_i,\bar{Z}_j)=g_{p\bar{j}}\Gamma_{\bar{\ell}i}^r\Gamma_{kr}^{p}=g_{p\bar{j}}\mu_{\bar{\ell}i}^r\Gamma_{kr}^{p} \,\, , \\
g(\nabla_{\bar Z_{\ell}}\nabla_{Z_k}Z_i,\bar{Z}_j)=g_{p\bar{j}}\Gamma_{ki}^r \Gamma_{{\bar{\ell}}r}^p=g_{p\bar{j}}\mu_{{\bar{\ell}}r}^p\Gamma_{ki}^r \,\, , \\
g(\nabla_{\mu(Z_k,\bar{Z}_{\ell})}Z_i,\bar{Z}_j)=g_{p\bar{j}}(\mu_{k\bar{\ell}}^r\Gamma_{ri}^p+\mu_{k\bar{\ell}}^{\bar{r}}\mu_{\bar{r}i}^{p}) \,\, .
\end{gathered}
$$
Thus, the second Chern-Ricci curvature $S$ takes the form
\begin{equation*}
S_{i\bar j}=g^{\bar{\ell}k}g_{p\bar{j}}(\mu_{\bar{\ell}i}^r\Gamma_{kr}^{p}-\mu_{{\bar{\ell}}r}^p\Gamma_{ki}^r -\mu_{k\bar{\ell}}^r\Gamma_{ri}^p-\mu_{k\bar{\ell}}^{\bar{r}}\mu_{\bar{r}i}^{p})\,\,,k
\end{equation*}
or, equivalently,
%\begin{equation*}
%S_{i\bar j}=g^{k\bar{\ell}}g_{p\bar{j}}(\mu_{\bar{\ell}i}^r(-g^{p\bar v}g_{r\bar q}\mu_{k\bar v}^{\bar q})-\mu_{{\bar{\ell}}r}^p(-g^{r\bar v}g_{i\bar q}\mu_{k\bar v}^{\bar q}) -\mu_{k\bar{\ell}}^r(-g^{p\bar v}g_{i\bar q}\mu_{r\bar v}^{\bar q})+\mu_{k\bar{\ell}}^{\bar{r}}\mu_{\bar{r}i}^{p})\,\,,
%\end{equation*}
\begin{equation*}
S_{i\bar j}=-g^{\bar{\ell}k}g_{p\bar{j}}(g^{p\bar v}g_{r\bar q}\,\mu_{k\bar v}^{\bar q}\mu_{\bar{\ell}i}^r-g^{r\bar v}g_{i\bar q}\,\mu_{k\bar v}^{\bar q}\mu_{{\bar{\ell}}r}^p -g^{p\bar v}g_{i\bar q}\,\mu_{r\bar v}^{\bar q}\mu_{k\bar{\ell}}^r+\mu_{k\bar{\ell}}^{\bar{r}}\mu_{\bar{r}i}^{p})\,\,.
\end{equation*}
On the other hand, since $T_{ij}=\nabla_{Z_i}Z_j-\nabla_{Z_j}Z_i-\mu(Z_i,Z_j)$, from \eqref{Gkis} we have
$$
T_{ij}^k= -g^{\bar qk} g_{j\bar a } \mu_{\bar qi}^{\bar a}+g^{\bar qk} g_{i\bar a } \mu_{\bar qj}^{\bar a}-\mu_{ij}^k
$$
and hence
\begin{equation}
T_{ij\bar{k}}=-g_{j\bar a } \mu_{\bar{k}i}^{\bar a}+g_{i\bar a } \mu_{\bar{k}j}^{\bar a}- g_{m\bar{k}}\mu_{ij}^m\,\,.
\label{Tijk} \end{equation}
Therefore, the explicit expression of the tensor $Q$ can be recovered from \eqref{defQ}, \eqref{defQi} and \eqref{Tijk}. Indeed,
$$\begin{aligned}
Q^1_{i\bar{j}}&=g^{\bar{\ell}k}g^{\bar{q}m}(-g_{k\bar a}\mu_{\bar q i}^{\bar a}+g_{i\bar a }\mu_{\bar q k}^{\bar a}-g_{v\bar q}\mu_{ik}^{v})(-g_{b\bar \ell}\mu_{m\bar j}^{b}+g_{b\bar j}\mu_{m\bar \ell}^{b}-g_{m\bar r}\mu_{\bar j\bar \ell}^{\bar r})\,\,,\\
Q^2_{i\bar{j}}&=g^{\bar{\ell}k}g^{\bar{q}m}(-g_{m\bar a }\mu_{\bar j k}^{\bar a}+g_{k\bar a}\mu_{\bar j m}^{\bar a}-g_{v\bar j}\mu_{km}^{v})(-g_{b\bar q}\mu_{i\bar \ell}^{b}+g_{b\bar \ell}\mu_{i\bar q}^{b}-g_{i\bar r}\mu_{\bar \ell\bar q}^{\bar r})\,\,,\\
Q^3_{i\bar{j}}&=g^{\bar{\ell}k}g^{\bar{q}m}(-g_{k\bar a }\mu_{\bar \ell i}^{\bar a}+g_{i\bar a }\mu_{\bar \ell k}^{\bar a}-g_{v\bar \ell}\mu_{ik}^{v})(-g_{b\bar q}\mu_{m\bar j}^{b}+g_{b\bar j}\mu_{m\bar q}^{b}-g_{m\bar r}\mu_{\bar j\bar q}^{\bar r})\,\,,\\
2Q^4_{i\bar{j}}&=g^{\bar{\ell}k}g^{\bar{q}m}\Big[(-g_{k\bar a }\mu_{\bar \ell m}^{\bar a}+g_{m\bar a}\mu_{\bar \ell k}^{\bar a}-g_{v\bar \ell}\mu_{mk}^{v})(-g_{b\bar j}\mu_{i\bar q}^{b}+g_{b\bar q}\mu_{i\bar j}^{b}-g_{i\bar r}\mu_{\bar q\bar j}^{\bar r})\\
&\qquad\qquad+(-g_{i\bar a}\mu_{\bar j m}^{\bar a}+g_{m\bar a}\mu_{\bar j i}^{\bar a}-g_{v\bar j}\mu_{mi}^{v})(-g_{b\bar \ell}\mu_{k\bar q}^{b}+g_{b\bar q}\mu_{k\bar \ell}^{b}-g_{k\bar r}\mu_{\bar q\bar \ell}^{\bar r})\Big]\,\,.
\end{aligned}
$$

\subsection{Complex model geometries} \hfill \par

In this subsection, we recall some basics about the geometry of locally homogeneous Hermitian manifolds. In particular, we focus on compact locally homogeneous Hermitian surfaces. \smallskip

A Hermitian manifold $(X,g)$ is {\it locally homogeneous} if the pseudogroup of local automorphisms of $(X,g)$ acts transitively on $X$, i.e. for any choice of $x, y \in X$ there exist neighborhoods $U_x,U_y \subset X$ of $x$ and $y$, respectively, and a holomorphic local isometry $f : U_x\rightarrow U_y$ such that $f(x) = y$. If in addiction $(X,g)$ is compact, then its universal Hermitian covering $(\wt{X},g)$ is {\it globally homogeneous} (see \cite{Singer}) and hence it admits a left coset presentation $\wt{X} = G/H$ for some closed subgroup $G \subset {\rm Aut}(\wt{X},g)$. Here, with a slight abuse of notation, we denote by $g$ both the Hermitian metric on $X$ and its pullback on the universal cover $\wt{X}$.

Motivated by this, we recall the following

\begin{definition}
A {\it complex model geometry of dimension $n$} is a pair $(\wt{X},G)$ given by a connected, simply-connected $n$-dimensional complex manifold $\wt{X}$ and a real connected Lie group $G$ such that: \begin{itemize}
\item[$\bcdot$] $G$ acts properly, transitively and almost-effectively by biholomorphisms on $\wt{X}$;
\item[$\bcdot$] $G$ contains a discrete subgroup $\Gamma \subset G$ with $\Gamma\backslash \wt{X}$ compact.
\end{itemize} If $G$ is a minimal group with such properties, then the complex model geometry is said to be {\it minimal}.
\end{definition}

Let $(\wt{X},G)$ be a complex model geometry. A Hermitian manifold $(X,g)$ has {\it geometric structure of type $(\wt{X},G)$} if $\wt{X}$ is the universal cover of $X$ and the pulled-back metric $g$ on $\wt{X}$ is invariant under the action of $G$. Of course, if $(X,g)$ has a geometric structure, then it is locally homogeneous. On the other hand, by the previous observation, any compact locally homogeneous Hermitian manifold has geometric structure of type $(\wt{X},G)$ for some {\it minimal} complex model geometry $(\wt{X},G)$. \smallskip

By the Riemann Uniformization Theorem, it is known that there exist exactly three minimal complex model geometries of dimension $1$, that are $$(\C,\C) \,\, , \quad (\C P^1,{\rm SU(2)}) \,\, , \quad (\C H^1,{\rm SU(1,1)}) \,\, .$$ Here, the group $G$ acts on the respective space $\wt{X}$ in the standard way.

Subsequently in \cite{Wall, Wall2} Wall classified all complex model geometries of dimension $2$. In particular, he proved the following

\begin{theorem}[\cite{Wall, Wall2}]\label{theor-Wall} If $(\wt{X},G)$ is a minimal complex model geometry of dimension 2, then one of the following cases occurs:
\begin{itemize} 
\item[i)] $(\wt{X},G)=(\wt{X}_1\times\wt{X}_2,G_1\times G_2)$ is the product of two complex model geometries of dimension $1$.
\item[ii)] $(\wt{X},G)=(\C P^2,{\rm SU}(3))$ or $(\wt{X},G)=(\C H^2,{\rm SU}(2,1))$, both considered endowed with the standard action of $G$ on $\wt{X}$.
\item[iii)] $\wt{X}=(G,J)$ where $G$ acts on itself by left translations and $J$ is a left-invariant complex structure.
\end{itemize} \label{model}\end{theorem}

\begin{rem} If $(\wt{X},G)$ is one of the model listed in (i) or (ii) above, then any Hermitian $G$-invariant metric on $\wt{X}$ is necessarily K\"ahler-Einstein.
\label{remKE} \end{rem}

\subsection{Gromov-Hausdorff convergence} \label{secGH} \hfill \par

We collect here some basic facts about Gromov-Hausdorff convergence of compact metric spaces. We refer to \cite[Sec. 7.3.2]{BBI} and \cite{Rong} for more details. 

\smallskip
Let $Z=(Z,d_Z)$ be a metric space and $X,Y \subset Z$ two compact subsets. The {\it Hausdorff distance} between $X$ and $Y$ is given by 
$$
{\rm dist}_{{}_{\rm H}}^Z(X,Y) := \inf\big\{\epsilon>0 : X \subset B_{\epsilon}(Y) \, , \,\, Y \subset B_{\epsilon}(X) \big\} \,\, ,
$$ 
where $B_{\epsilon}(X):=\{x \in Z : d_Z(x,X)<\epsilon\}$ is the $\epsilon$-tube of $X$ in $Z$. The pair
$$
\left(\{\text{compact subsets of $Z$}\},{\rm dist}_{{}_{\rm H}}^Z\right)
$$ 
is also a metric space and it is compact if and only if $Z$ is compact as well. \smallskip

Let now $X=(X,d_X)$, $Y=(Y,d_Y)$ be two compact metric spaces. The {\it Gromov-Hausdorff distance} between $X$ and $Y$ is defined as 
$$
{\rm dist}_{{}_{\rm GH}}(X,Y) := \inf\left\lbrace{\rm dist}_{{}_{\rm H}}^Z\left(\phi_1(X),\phi_2(Y)\right) \right\rbrace \,\, ,
$$
where the infimum is taken with respect to all metric spaces $Z$ and all pairs $(\phi_1,\phi_2)$ of isometric embeddings $\phi_1: X \to Z$ and $\phi_2: Y \to Z$. Letting $\mathcal{X}$ denote the set of isometric classes of compact metric spaces, it turns out that $\left(\mathcal{X},{\rm dist}_{{}_{\rm GH}}\right)$ is a complete metric space. Therefore, given a one-parameter family $\{X_t\}_{t \in [0,T)}$ and an element $Y$ both in $\mathcal{X}$, whenever $\lim_{t \to T^-}{\rm dist}_{{}_{\rm GH}}(X_t,Y)=0$ we write 
$$ 
X_t \GHa Y \quad \text{ as } \, t \to T^-
$$ 
and we say that $X_t$ {\em convergences in the Gromov-Hausdorff topology} to $Y$. \smallskip

Finally, a {\it GH $\epsilon$-approximation} between two metric spaces $X,Y \in \mathcal{X}$, with $\epsilon>0$, is a pair of non-necessarily continuous maps $\varphi: X \to Y$ and $\psi: Y \to X$ satisfying for any $x,x' \in X$ and $y,y' \in Y$
$$
\begin{gathered}
\big|d_X(x,x')-d_Y(\varphi(x),\varphi(x'))\big|<\epsilon \,\, , \quad d_X(x,(\psi\circ\varphi)(x))<\epsilon \,\, ,\\
\big|d_Y(y,y')-d_X(\psi(y),\psi(y'))\big|<\epsilon \,\, , \quad d_Y(y,(\varphi\circ\psi)(y))<\epsilon \,\, .
\end{gathered}
$$ Remarkably, if there exists a GH $\epsilon$-approximation $(\varphi,\psi)$ between $X$ and $Y$, then ${\rm dist}_{{}_{\rm GH}}(X,Y) \leq \frac32\epsilon$ (see e.g. \cite[Lemma 1.3.3]{Rong}).

\section{HCF tensor on complex model geometries} \label{model_sub}

The aim of this section is to compute the HCF tensor $K$ of any $2$-dimensional complex model geometry $(\wt{X},G)$ endowed with an invariant metric $g$. By means of Remark \ref{remKE}, we will restrict our discussion to those minimal complex model geometries arising from (iii) in Theorem \ref{model}. Hence, following \cite[Sec. 2.2]{Bol}, we list below all the connected, simply-connected real 4-dimensional Lie groups which admits a left-invariant complex structure, their compact quotients according to Enriques-Kodaira classification and their HCF tensors. We mention here that all the computations were made with the help of the software Maple. \smallskip

In the following, given a connected, simply connected $4$-dimensional real Lie groups $(G,J)$ equipped with a left-invariant complex structure, we will consider a fixed left-invariant $(1,0)$-frame $\{Z_1,Z_2\}$ and we will denote by $\{\zeta^1,\zeta^2\}$ its dual frame. This allows us to write any left-invariant Hermitian metric $g$ on $(G,J)$ in the form
\begin{equation}
g= x\, \zeta^1\odot \bar \zeta^1 + y\, \zeta^2\odot \bar \zeta^2 + z\, \zeta^1\odot \bar \zeta^2 + \bar z\, \zeta^2\odot \bar \zeta^1 \,\, , \label{leftinvmetric}
\end{equation} with $x,y \in \R_{>0}$, $z \in \C$ and $xy-|z|^2>0$.

\begin{remark} We refer to the Appendix for an explicit computation of the HCF tensors written in the following. 
\end{remark}

\subsubsection*{Complex tori} \hfill \par

The Lie group is $G=\R^4$, which is abelian and admits a unique left-invariant complex structure $J_{\rm st}$. In this case, the HCF tensor of any left-invariant metric on $\C^2=(\R^4,J_{\rm st})$ is just $K=0$. Compact quotients of $\C^2$ are complex tori.

\subsubsection*{Hyperelliptic surfaces} \hfill \par

The Lie group is $G=\widetilde{\rm SE}(2)\times\R$, where $\widetilde{\rm SE}(2)$ is the universal cover of the special Euclidean group ${\rm SE}(2) := {\rm SO}(2) \ltimes \R^2$. It admits a unique left-invariant complex structure $J$ and the structure constants $\mu$ of its complexified Lie algebra are $$
\mu(Z_1, Z_2)=Z_1 \,\, , \quad \mu(Z_1,\bar Z_2)=-Z_1 \,\, .
$$ The HCF tensor of a left invariant Hermitian metric on $\big(\widetilde{\rm SE}(2)\times\R,J\big)$ is given by
$$
K_{1\bar 1}  = \frac {x^2|z|^2}{(xy-|z|^2)^2} \,\, , \quad K_{2\bar 2}  = \frac {|z|^4}{(xy-|z|^2)^2} \,\, , \quad K_{1\bar 2}  = \frac {x^2yz}{(xy-|z|^2)^2} \,\, .
$$
Compact quotients of $\big(\widetilde{\rm SE}(2)\times\R,J\big)$ are hyperelliptic surfaces, which admit K\"ahler metrics.

\subsubsection*{Hopf surfaces} \hfill \par

The Lie group is $G={\rm SU(2)}\times\R$. It admits a one-parameter family $J_{\lambda}$ of left-invariant complex structures, with $\lambda \in \R$, and with respect to $J_\lambda$ the structure constants $\mu=\mu_{\lambda}$ of its complexified Lie algebra are
$$
\mu(Z_1, Z_2)=Z_2 \,\, , \quad \mu(Z_1,\bar Z_2)=-\bar Z_2 \,\, , \quad \mu(Z_2,\bar Z_2)=(-1+\sqrt{-1}\lambda)Z_1+(1+\sqrt{-1}\lambda)\bar Z_1 \,\, .
$$
The HCF tensor of a left-invariant Hermitian metric on $\big({\rm SU(2)}\times\R,J_{\lambda}\big)$ is given by
\begin{equation*}\left.\begin{aligned}
K_{1\bar 1} &= \frac {x^4(1+\lambda^2)+|z|^2(2x^2+|z|^2)}{(xy-|z|^2)^2} \\
K_{2\bar 2} &= \frac {(1+\lambda^2)x^2|z|^2+2(xy-|z|^2)^2+|z|^2(y^2+2|z|^2)-2(1+\lambda^2)x^2(xy-|z|^2)}{(xy-|z|^2)^2}  \\
K_{1\bar 2} &= \frac {xz(\lambda^2x^2+(x+y)^2)}{(xy-|z|^2)^2}
\end{aligned}\right. . \end{equation*} 
Compact quotients of $\big({\rm SU(2)}\times\R,J_{\lambda}\big)$ are Hopf surfaces, which are non-K\"ahler.

\subsubsection*{Non-K\"ahler properly elliptic surfaces} \hfill \par

The Lie group is $G=\widetilde{\rm SL}(2,\R)\times\R$, where $\widetilde{\rm SL}(2,\R)$ is the universal cover of ${\rm SL}(2,\R)$. It admits a one-parameter family $J_{\lambda}$ of left-invariant complex structure, with $\lambda \in \R$, with respect to which the structure constants $\mu=\mu_{\lambda}$ of its complexified Lie algebra are
$$
\mu(Z_1, Z_2)=\sqrt{-1}Z_1 \,\, , \quad \mu(Z_1,\bar Z_2)=\sqrt{-1} \bar Z_1 \,\, , \quad \mu(Z_1,\bar Z_1)=(-\lambda+\sqrt{-1})Z_2+(\lambda+\sqrt{-1})\bar Z_2 \,\, .
$$
The HCF tensor of a left-invariant Hermitian metric on $\big(\widetilde{\rm SL}(2,\R)\times\R,J_{\lambda}\big)$ is given by
\begin{equation*}\left.\begin{aligned}
K_{1\bar 1} &= \frac {(1+\lambda^2)y^2|z|^2-2(xy-|z|^2)^2+|z|^2(x^2-2|z|^2)-2(1+\lambda^2)y^2(xy-|z|^2)}{(xy-|z|^2)^2}\\
K_{2\bar 2} &= \frac {\lambda^2y^4+(y^2-|z|^2)^2}{(xy-|z|^2)^2}  \\
K_{1\bar 2} &= \frac {yz(\lambda^2y^2+(x-y)^2)}{(xy-|z|^2)^2}
\end{aligned}\right. . \end{equation*}
Compact quotients of $\big(\widetilde{\rm SL}(2,\R)\times\R,J_{\lambda}\big)$ are non-K\"ahler properly elliptic surfaces.

\subsubsection*{Primary Kodaira surfaces} \hfill \par

The Lie group is $G= \R \times {\rm H}_3(\R)$, where ${\rm H}_3(\R)$ is the three-dimensional real Heisenberg group. It admits a unique left-invariant complex structure $J$ and the structure constants $\mu$ of its complexified Lie algebra are
$$
\mu(Z_1,\bar Z_1)=\sqrt{-1}(Z_2+\bar Z_2) \,\, .
$$
The HCF tensor of a left-invariant Hermitian metric on $\big(\R \times {\rm H}_3(\R),J\big)$ is 
$$
K_{1\bar 1} = \frac {-2y^2(xy-|z|^2)+y^2|z|^2}{(xy-|z|^2)^2} \,\, , \quad K_{2\bar 2} = \frac {y^4}{(xy-|z|^2)^2} \,\, , \quad K_{1\bar 2} = \frac {y^3z}{(xy-|z|^2)^2} \,\, .
$$
Compact quotients of $\big(\R \times {\rm H}_3(\R),J\big)$ are primary Kodaira surfaces, which are non-K\"ahler.

\subsubsection*{Secondary Kodaira surfaces} \hfill \par

The Lie group is $G= \R \ltimes {\rm H}_3(\R)$. It admits two different left-invariant complex structure $J_{\pm}$ and the structure constants $\mu=\mu_{\pm}$ of its complexified Lie algebra are
$$
\mu(Z_1, Z_2)=\varepsilon Z_1\,\,,\quad \mu(Z_1,\bar Z_2)=-\varepsilon Z_1\,\,,\quad \mu(Z_1,\bar Z_1)=-\varepsilon\sqrt{-1}(Z_2+\bar Z_2) \,\, , \quad \text{ with } \varepsilon = \pm 1 \,\, .
$$
The HCF tensor of a left-invariant Hermitian metric on $\big(\R \ltimes {\rm H}_3(\R),J_{\pm}\big)$ is given by
$$
K_{1\bar 1} = \frac {|z|^2(x^2+y^2)-2y^2(xy-|z|^2)}{(xy-|z|^2)^2} \,\, , \quad K_{2\bar 2} = \frac {y^4+|z|^4}{(xy-|z|^2)^2} \,\, , \quad K_{1\bar 2} = \frac {yz(x^2+y^2)}{(xy-|z|^2)^2} \,\, .
$$
Compact quotients of $\big(\R \ltimes {\rm H}_3(\R),J_{\pm}\big)$ are secondary Kodaira surfaces, which are non-K\"ahler.

\subsubsection*{Inoue surfaces of type $S^0$} \hfill \par

The group $G={\rm Sol}_0^4$ is a solvable 4-dimensional real Lie group which admits a two-parameter family $J_{a,b}$ of left-invariant complex structures, where $a,b \in \R$, and with respect to $J_{a,b}$ the structure constants $\mu=\mu_{a,b}$ of its complexified Lie algebra are
\begin{equation*}
\mu(Z_1, Z_2)=-(b+\sqrt{-1}a)Z_1 \,\, , \quad \mu(Z_1,\bar Z_2)=(b+\sqrt{-1}a)Z_1 \,\, , \quad \mu(Z_2,\bar Z_2)=-2\sqrt{-1}a(Z_2+\bar Z_2) \,\, .
\label{bracketS0} \end{equation*}
The HCF tensor of a left-invariant Hermitian metric on $\big({\rm Sol}_0^4,J_{a,b}\big)$ is given by
$$
\left.\begin{aligned}
K_{1\bar 1} &= \frac {x^2|z|^2(b^2+9a^2)}{(xy-|z|^2)^2} \\
K_{2\bar 2} &= \frac {|z|^4(a^2+b^2)+16|z|^2a^2xy-8a^2x^2y^2}{(xy-|z|^2)^2}\\
K_{1\bar 2} &= \frac {x^2yz(b^2+9a^2)}{(xy-|z|^2)^2} 
\end{aligned} \right. .
$$
Notice that $\big({\rm Sol}_0^4,J_{a,b}\big)$ does not always admit a co-compact lattice. When such a lattice does exist, the quotient is an Inoue surface of type $S^0$, which is non-K\"ahler.

\subsubsection*{Inoue surfaces of type $S^{\pm}$} \hfill \par

The group $G={\rm Sol}_1^4$ is a solvable 4-dimensional real Lie group which admits two different left-invariant complex structure $J_{1,2}$. The structure constants $\mu=\mu_{1}$ of the complexified Lie algebra of $\big({\rm Sol}_1^4,J_{1}\big)$ are 
$$
\mu(Z_1, Z_2)=- Z_2 \,\, , \quad \mu(Z_1,\bar Z_2)=- Z_2 \,\, , \quad \mu(Z_1,\bar Z_1)=-Z_1+\bar Z_1\,\,,
$$ 
and the HCF tensor of a left-invariant Hermitian metric on $\big({\rm Sol}_1^4,J_{1}\big)$ is given by
\begin{equation*}
K_{1\bar 1} = -3-\frac {|z|^2(z-\bar z)^2}{(xy-|z|^2)^2}\,\,,\quad K_{2\bar 2} = -\frac {y^2(z-\bar z)^2}{(xy-|z|^2)^2}\,\,,\quad K_{1\bar 2} = \frac {y(z( \bar z^2 -z^2)-2xy(\bar z- z))}{(xy-|z|^2)^2}\,.
\end{equation*} 

\noindent On the other hand, the structure constants $\mu=\mu_{2}$ of the complexified Lie algebra of $\big({\rm Sol}_1^4,J_{2}\big)$ are 
$$
\mu(Z_1, Z_2)=- Z_2\,\,,\quad \mu(Z_1,\bar Z_2)=- Z_2\,\,,\quad \mu(Z_1,\bar Z_1)=-Z_1+\bar Z_1+Z_2-\bar Z_2\,\,,
$$
and the HCF tensor of a left-invariant Hermitian metric on $\big({\rm Sol}_1^4,J_{2}\big)$ is given by
$$
\left.\begin{aligned}
K_{1\bar 1} &= -3-\frac {|z|^2(z-\bar z)^2+2y^2(xy-|z|^2)-y^2|z|^2}{(xy-|z|^2)^2} \\
K_{2\bar 2} &= -\frac {y^2((z-\bar z)^2-y^2)}{(xy-|z|^2)^2} \\
K_{1\bar 2} &= \frac {y(z( \bar z^2 -z^2)-2xy(\bar z- z)+y^2z)}{(xy-|z|^2)^2}
\end{aligned}\right.. 
$$

\noindent Compact quotients of $\big({\rm Sol}_1^4,J_{1}\big)$ are Inoue surfaces of type $S^{\pm}$, while compact quotient of $\big({\rm Sol}_1^4,J_{2}\big)$ are Inoue surfaces of type $S^{+}$. In both cases, these surfaces are non-K\"ahler.

\section{HCF on locally homogeneous surfaces} \label{mainsec}

In this section we study the behavior of locally homogeneous solutions to the HCF on the family of compact complex surfaces we listed in Section \ref{model_sub}. Furthermore, whenever a solution to the HCF is immortal, we determine the Gromov-Hausdorff limit of its normalization $(1{+}t)^{-1}g(t)$ as $t \to +\infty$. \smallskip

Let $X$ be a compact complex surface covered by a connected, simply-connected 4-dimensional real Lie group $G$ and $\Gamma\subset G$ a co-compact lattice such that $X=\Gamma\backslash G$. By construction, all left-invariant tensor fields on $G$ factorizes through $X$. This yields a one-to-one correspondence between locally homogeneous solutions to the HCF on $X$ and solutions to the corresponding ODE on $G$
\begin{equation*}
\tfrac{d}{dt}g(t)=-K(g(t))\,\,,\qquad g(0)=g_0\,\,,
\end{equation*}
where $g_0$ denotes the pull-back of the starting metric on $G$. Nonetheless, since the standard left-action of $G$ on itself does not always factorize through $X=\Gamma\backslash G$, the quotient $\Gamma\backslash G$ is not globally $G$-homogeneous in general.

\medskip
\noindent {\bf Notation.} Any left-invariant Hermitian metric $g$ on $(G,J)$ will be considered in the form of \eqref{leftinvmetric}. For the sake of shortness, we set $D:=xy-|z|^2$ and $u:=|z|^2$.

\subsection{Hyperelliptic surfaces} \hfill \par

The HCF on $\big(\widetilde{\rm SE}(2)\times\R,J\big)$ reduces to the following ODEs system:

\begin{equation*}
\dot x =- \frac {x^2u}{D^2}\,\,,\qquad \dot y=- \frac {u^2}{D^2}\,\,,\qquad \dot u=-2 \frac {x^2yu}{D^2}\,.
\end{equation*}

\begin{prop}\label{prop_hyper} Let $g_0$ be a locally homogeneous Hermitian metric on a hyperelliptic surface $X$. Then, the solution $g(t)$ to the HCF starting from $g_0$ exists for all $t\geq0$. Moreover
$$
\left(X,(1{+}t)^{-1}g(t)\right) \GHa \text{\rm \{point\}}\quad \text{as}\,\,\, t\to \infty\,\,.
$$
\end{prop}

\begin{proof} 
A direct computation yields that
$$
\dot{D}=\frac{xu}{D}\geq0 \,\, ,
$$ i.e. the determinant of $g(t)$ is always increasing. On the other hand, since all $x,y,u$ decrease, the first claim follows. The last claim follows directly from the fact that 
$$
(1{+}t)^{-1}x(t) \, , \,\, (1{+}t)^{-1}y(t) \, , \,\, (1{+}t)^{-1}u(t) \rightarrow 0
$$ 
as $t \rightarrow +\infty$.
\end{proof}

It is easy to show that a left-invariant metric $g$ on $\big(\widetilde{\rm SE}(2)\times\R,J\big)$ is K\"ahler if and only if $z=0$. Indeed, by a direct computation, one gets
$$
d\omega=-z\, \zeta^{1}\wedge\zeta^2\wedge\bar\zeta^{ 2}+\bar z\, \zeta^{2}\wedge\bar \zeta^{1}\wedge\bar \zeta^{  2}\,.
$$
Moreover, in that case it is also flat and so we get

\begin{cor} Any locally homogeneous solution $g(t)$ to the HCF on a hyperelliptic surface $X$ converges exponentially fast to a flat K\"ahler metric $g_{\infty}$.
\end{cor}

\begin{proof} We recall that $g(t)$ is immortal and $\dot{D}(t)>0$, $x(t) < x_0$, $y(t) < y_0$, $u(t) < u_0$ for any $t\geq0$. Notice that 
$$
\dot u\leq -2 \frac{u}{y_0} \,\, ,
$$
which implies $u(t)\leq u_0e^{-\frac2{y_0}t}$ for all $t\geq0$. Finally, since 
$$
\lim_{t \to +\infty}D(t) = D_{\infty} \in (D_0, +\infty)\,\,,
$$ 
it comes that $x(t) \to x_{\infty} \in (0,x_0)$ and $y(t) \to y_{\infty} \in (0,y_0)$ as $t \to +\infty$.
\end{proof}

\subsection{Hopf surfaces} \hfill \par

The HCF on $\big({\rm SU(2)}\times\R,J_{\lambda}\big)$ reduces to the ODEs system

\begin{equation} \label{Hopf}
\left.\begin{aligned}
\dot{x}&=-\frac{cx^4+u(2x^2+u)}{D^2} \\
\dot{y}&=-2+\frac{2cx^2D-cx^2u-u(y^2+2u)}{D^2} \\
\dot{u}&=-2\frac{xu(cx^2+2xy+y^2)}{D^2}
\end{aligned}\right.,
\end{equation} 
with $c:=1+\lambda^2$.

\begin{prop}\label{prop_hopf} Let $g_0$ be a locally homogeneous Hermitian metric on a Hopf surface $X$. Then, the solution $g(t)$ to the HCF starting from $g_0$ develops a finite extinction time $T<\infty$ and $(X,g(t))$ collapses as $t \rightarrow T^-$.
\end{prop}

\begin{proof}
Let $T\in(0,+\infty]$ be the maximal existence time of the flow. Then for any $t \in [0,T)$ we have
\begin{equation} \label{estDuy***}
\begin{aligned}
\dot{D}=&\frac{c\,x^3-2x^2y+(4x+y)u}{D}\,\,, \\ 
\dot{x}<0 \,\,, \quad \dot{u}&<0 \quad\Longrightarrow\quad x(t)\leq x_0 \,\,, \quad u(t)\leq u_0 \,. 
\end{aligned}\end{equation}
Let us suppose by contradiction that $T=+\infty$. Then it necessarily holds
\begin{equation} \label{absestxudot}
\begin{aligned} 
\lim_{t\rightarrow+\infty}\dot{x}(t)=0 \quad&\Longrightarrow\quad \lim_{t\rightarrow+\infty}(c-1)\Big(\frac{x^2}{D}\Big)^2=\lim_{t\rightarrow+\infty}\frac{x^2+u}{D}=0 \,\, \\ 
\lim_{t\rightarrow+\infty}\dot{u}(t)=0 \quad&\Longrightarrow\quad \lim_{t\rightarrow+\infty}\frac{u}{x}(c-1)\Big(\frac{x^2}{D}\Big)^2=\lim_{t\rightarrow+\infty}xu\Big(\frac{x+y}{D}\Big)^2=0 \,\, . 
\end{aligned} \end{equation}
On the other hand
$$
\dot{y}+2 = \frac{2cx^2D-cx^2u-u(y^2+2u)}{D^2} \leq 2c\frac{x^2}{D} \leq 2c\frac{x^2+u}{D} \,\, ,
$$ 
and so by means of \eqref{absestxudot}
$$
\lim_{t\rightarrow +\infty}\dot{y}(t) \leq -2
$$ 
which is absurd. Thus $g(t)$ develops a finite time singularity $T<\infty$. In order to prove the last claim, let us suppose by contradiction that $D\to\infty$ as $t\to T^-$. Then
$$
\lim_{t\to T^-} \dot x(t)=0 \qquad\text{and}\qquad \lim_{t\to T^-} \dot y(t)<-2\,\,,
$$
this in turn imply $\lim_{t\to T^-} D\neq \infty$, which is not possible. On the other hand, since the solution cannot be extended over $t=T$, the limit $\lim_{t\to T^-} D$ cannot be positive and finite. Therefore, $\lim_{t\to T^-}D=0$ and the thesis follows.
\end{proof}

Next, we exhibit an explicit solution to the HCF starting from a diagonal metric on $\big({\rm SU(2)}\times\R,J_{\lambda}\big)$. 

\begin{example}
Let $g_0$ be a left-invariant diagonal Hermitian metric on $\big({\rm SU(2)}\times\R,J_{\lambda}\big)$. Then, the ODEs system \eqref{Hopf} reduces to
\begin{equation}\label{diagHopf} \begin{aligned}
\dot{x}=-c\frac{x^2}{y^2}\,,\qquad \dot{y}=-2\frac{y-cx}{y}\,\,. 
\end{aligned}\end{equation}
It is worth noting that
\begin{equation} \label{derII}
\ddot{x}=-4c\frac{x^2}{y^2}\Big(y-\frac32cx\Big) \,\, , \qquad \ddot{y}=+4c\frac{x}{y^3}\Big(y-\frac32cx\Big) \,\,. 
\end{equation} 
Now suppose that $y_0=\frac32cx_0$ and that the solution to \eqref{diagHopf} starting from $g_0$ satisfies
$$
y(t)=\frac32c\,x(t) \quad \text{ for all } t \in [0,T) \,\, .
$$
Then by \eqref{derII} we would get 
$$
\ddot{x}(t)=\ddot{y}(t)=0\,\,,
$$  
which in turn implies
\begin{equation}\label{sol_hopf}
x(t)=x_0+kt \,\, , \quad  y(t)=\tfrac32cx_0+\tfrac32ckt
\end{equation}
for some $k \in \R$. A direct computation yields that \eqref{sol_hopf} solves \eqref{diagHopf} if and only if $k=-\frac4{9c}$. Notice that the maximal existence time for this explicit solution is 
$
T=\frac 94cx_0\,.
$ 
\end{example}

\subsection{Non-K\"ahler properly elliptic surfaces} \hfill \par

The HCF on $\big(\widetilde{\rm SL}(2,\R)\times\R,J_{\lambda}\big)$ reduces to the ODEs system

\begin{equation}\label{propell}
\left.\begin{aligned}
\dot{x}=&2+\frac{2cy^2D-cy^2u-ux^2+2u^2}{D^2} \\
\dot{y}=&-\frac{cy^4-2y^2u+u^2}{D^2} \quad \quad \quad \quad \quad \quad \\
\dot{u}=&-2\frac{yu(x^2-2xy+cy^2)}{D^2}
\end{aligned} \right.,
\end{equation} 
with $c:=1+\lambda^2$.

\begin{prop}\label{prop_ell} Let $g_0$ be a locally homogeneous Hermitian metric on a non-K\"ahler properly elliptic surface $X$. Then, the solution $g(t)$ to the HCF starting from $g_0$ exists for all $t\geq0$. In particular, $x(t)\sim2t$ and $y(t)<y_0$, $u(t)<u_0$ for any $t>0$.
\end{prop}

\begin{proof} 

Let $T\in(0,+\infty]$ be the maximal existence time of the flow. Then, for any $t \in [0,T)$, we have
\begin{equation} \label{estDuy**}
\begin{aligned} \dot{D}=\frac{cy^3+2y(D-u)+xu}{D} \,\,,& \\ 
\dot{y}<0 \,\,, \quad \dot{u}<0 \quad\Longrightarrow\quad y(t)\leq y_0 \,\,,& \quad u(t)\leq u_0 \,\, . 
\end{aligned} \end{equation}
We prove now that $\dot{D}(t)>0$ for any $t \in [0,T)$. Let us suppose by contradiction that there exists $t_* \in [0,T)$ such that $\dot{D}(t_*)\leq0$. Then using \eqref{estDuy**} we get
\begin{equation} \label{absest} 
-x(t_*)u(t_*)\geq cy(t_*)^3-2y(t_*)(u(t_*)-D(t_*)) \,\,. 
\end{equation} 
On the other hand, since $D(t)=x(t)y(t)-u(t)$ and $\dot{u}(t_*)<0$, it necessarily holds 
\begin{equation} \label{dotxy<0} 
\dot{x}(t_*)y(t_*)+x(t_*)\dot{y}(t_*)\leq 0 \,\,. 
\end{equation} 
Moreover, by \eqref{absest} and a straightforward computation we get
\begin{equation}
D(t_*)^2\dot{x}(t_*)y(t_*) \geq 4D(t_*)^2y(t_*)+3cy(t_*)^3D(t_*)
\label{Dxyabs1} \end{equation}
and 
\begin{equation}
D(t_*)^2x(t_*)\dot{y}(t_*)\geq4y(t_*)u(t_*)D(t_*)-cy(t_*)^3D(t_*)\,\,.
\label{Dxyabs2} \end{equation}
Finally, \eqref{dotxy<0}, \eqref{Dxyabs1} and \eqref{Dxyabs2} imply
$$
4D(t_*)y(t_*)+2cy(t_*)^2+4y(t_*)u(t_*) \leq D(t_*)(\dot{x}(t_*)y(t_*)+x(t_*)\dot{y}(t_*)) \leq 0
$$ 
which is not possible. Hence the determinant $D$ satisfies
\begin{equation} 
\dot{D} > 0 \quad\Longrightarrow\quad D(t) \geq D_0 \quad \text{ for all } t \in [0,T) \,\, .\label{Dcresc}
\end{equation} 
On the other hand, it holds
\begin{equation} \label{estx**} 
\dot{x} \leq 2+\frac{2cy^2D+2u^2}{D^2} \leq 2\Big(1+c\frac{y_0^2}{D_0}+\frac{u_0^2}{D_0^2}\Big) \quad\Longrightarrow\quad x(t)\leq 2\Big(1+c\frac{y_0^2}{D_0}+\frac{u_0^2}{D_0^2}\Big)t+x_0
\end{equation} and hence \eqref{estDuy**}, \eqref{Dcresc} and \eqref{estx**} imply  $T=+\infty$.

We are now ready to prove the second part of the proposition, that is, $x(t)\sim 2t$ as $t\to\infty$. To do this, we use again a contradiction argument. Let us denote with
$$
u_{\infty}:=\lim_{t \rightarrow +\infty}u(t)\,\,,\qquad y_{\infty}:=\lim_{t \rightarrow +\infty}y(t)\,\,,
$$ and suppose by contradiction that $u_{\infty}>0$. Since
\begin{equation*} \label{estdotyu**}
\begin{aligned} 
\lim_{t\rightarrow+\infty}\dot{y}(t)=0 \quad&\Longrightarrow\quad \lim_{t\rightarrow+\infty}(c-1)\Big(\frac{y^2}{D}\Big)^2=\lim_{t\rightarrow+\infty}\frac{y^2-u}{D}=0 \,\,,\\ 
\lim_{t\rightarrow+\infty}\dot{u}(t)=0 \quad&\Longrightarrow\quad \lim_{t\rightarrow+\infty}\frac{u}{y}(c-1)\Big(\frac{y^2}{D}\Big)^2=\lim_{t\rightarrow+\infty}yu\Big(\frac{x-y}{D}\Big)^2=0 \,\,,
\end{aligned}\end{equation*} 
we have by means of \eqref{estDuy**} 
\begin{equation}\label{ulim>0} 
\lim_{t \rightarrow +\infty}\frac{y(x-y)}{D}=\lim_{t\rightarrow+\infty}\frac{y^2-u}{D}=0 \quad\Longrightarrow\quad \lim_{t \rightarrow +\infty}\frac{\frac{y}{x}-\frac{u}{xy}}{1-\frac{u}{xy}}= \lim_{t \rightarrow +\infty}\frac{1-\frac{y}{x}}{1-\frac{u}{xy}}=0 \,\,.
\end{equation} 
In view of \eqref{ulim>0}, we have two cases depending on whether $\lim_{t\to\infty}|1-\frac{u}{xy}|$ is bounded or not. If we suppose that $\lim_{t\to\infty}|1-\frac{u}{xy}|<\infty$, then
$$
\lim_{t \rightarrow +\infty}xy=u_{\infty}\quad \text{and}\quad \lim_{t \rightarrow +\infty}D=0 \,\, .
$$
On the other hand, if $\lim_{t\to\infty}|1-\frac{u}{xy}|=\infty$, then
$$
\lim_{t \rightarrow +\infty}xy=0\quad \text{and}\quad \lim_{t \rightarrow +\infty}D=-u_{\infty} \,\, .
$$
Since both cases lead to an absurd, it comes
\begin{equation}\label{u_inf}
u_{\infty}=0\,\,.
\end{equation}
Let us now suppose by contradiction that ${x(t) \rightarrow x_{\infty}<+\infty}$ as $t \rightarrow +\infty$. Then $D(t) \rightarrow D_{\infty}=x_{\infty}y_{\infty} \in (D_0,+\infty)$ as $t \rightarrow +\infty$ and therefore it must holds $x_{\infty}>0$. By means of \eqref{estDuy**}
$$
\lim_{t \rightarrow +\infty}\dot{D}(t)=0 \quad\Longrightarrow\quad cy_{\infty}^3+2y_{\infty}D_{\infty}=0 \quad\Longrightarrow\quad y_{\infty}=0 \quad\Longrightarrow\quad D_{\infty}=0
$$
which is not possible. Therefore $x(t)\to\infty$ as $t\to \infty$. On the other hand, we have
$$
\dot{x}=2+2c\frac{y^2}{D}-cu\left(\frac{y}{D}\right)^2-\frac{ux^2}{D^2}+2\frac{u^2}{D^2}
$$
and, since 
\begin{equation}\label{conv_var}
\frac{y^2}{D}  \rightarrow 0 \,\, , \quad u\left(\frac{y}{D}\right)^2 \rightarrow 0 \,\, , \quad u\left(\frac{x}{D}\right)^2 \rightarrow 0 \,\, , \quad \left(\frac{u}{D}\right)^2 \rightarrow 0 \,\, ,
\end{equation}
the claim follows. Indeed, if $y_{\infty}>0$, then $D_{\infty} = +\infty$ and hence \eqref{conv_var} follows. Now, let us assume that $y_{\infty}=0$. Since $D_{\infty}>D_0 >0$, we get 
$$
\frac{y^2}{D}  \rightarrow 0 \,\, , \quad u\left(\frac{y}{D}\right)^2 \rightarrow 0 \,\, , \quad \left(\frac{u}{D}\right)^2 \rightarrow 0 \,\, .
$$
Moreover, $u\left(\frac{x}{D}\right)^2 \sim \frac{u}{y^2}$ and 
$$
\frac{d}{dt}\Big(\frac{u}{y^2}\Big) = 2\Big(\frac{u}{y^2}\Big){\cdot} \frac1y\Big({-}1+2\frac{y^2-u}{D}\Big) \,\, .
$$ 
Hence, there exist $C>0$ and $t_*>0$ such that 
$$
\frac{d}{dt}\Big(\frac{u}{y^2}\Big) \leq -2C\Big(\frac{u}{y^2}\Big) \quad \text{ for any }t\geq t_*\,\, .
$$ 
This in turns implies 
$$
\Big(\frac{u(t)}{y(t)^2}\Big) \leq \Big(\frac{u(t_*)}{y(t_*)^2}\Big)e^{-2C(t-t_*)} \quad \text{ for any }t\geq t_*\,\, ,
$$ 
and hence $u\left(\frac{x}{D}\right)^2 \to 0$.
\end{proof}

In view of this result it comes the following

\begin{prop}\label{prop_ell_GH} Let $X$ be a non-K\"ahler properly elliptic surface and $g(t)$ be a locally homogeneous solution to the HCF on $X$. Then
$$
\left(X,(1{+}t)^{-1}g(t)\right) \GHa \left(C, g_{{}_{\rm KE}}\right) \quad \text{ as } t \to \infty\,\,,
$$ 
where $C$ is the base curve of $X$ and $g_{{}_{\rm KE}}$ is the K\"ahler-Einstein metric on $C$ with ${\rm Ric}(g_{{}_{\rm KE}})=-g_{{}_{\rm KE}}$.
\end{prop}

The proof of this statement follows the same arguments used in \cite[Thm 1.6 (c)]{TW}. For this reason, we just recall the main points.

\begin{proof} By definition, a properly elliptic surface is a compact complex surface $X$ with Kodaira dimension $\kappa(X)=1$ and first Betti number $b_1(X)$ odd admitting an elliptic fibration $\pi: X \to C$ over a compact complex curve $C$ of genus ${\rm g}(C)\geq2$. Moreover, by the Riemann Uniformization Theorem, $C$ admits a unique K\"ahler-Einstein metric $g_{{}_{\rm KE}}$ with ${\rm Ric}(g_{{}_{\rm KE}})=-g_{{}_{\rm KE}}$. Note that, this metric also satisfies $\pi^*g_{{}_{\rm KE}}=2 \zeta^1\otimes \bar \zeta^1$.

On the other hand, the fibers of the elliptic fibration $\pi: X \to C$ are spanned by the real and imaginary parts of $Z_2$, which shrinks to zero along $(1{+}t)^{-1}g(t)$ as $t\to \infty$. Therefore, if we consider a not necessarily continuous function $f: C \to S$ satisfying $\pi \circ f= {\rm id}$, then for any $\epsilon>0$ there exists $t_*(\epsilon)>0$ such that $(\pi,f)$ is a GH $\epsilon$-approximation between $\left(X,(1{+}t)^{-1}g(t)\right)$ and $(C, g_{{}_{\rm KE}})$ for any $t>t_*(\epsilon)$. This concludes the proof. \end{proof}

\subsection{Primary Kodaira surfaces} \hfill \par

The HCF on $\big(\R \times {\rm H}_3(\R),J\big)$ reduces to the ODEs system

\begin{equation}\label{IKod} 
\dot{x}=\frac{2y^2D-y^2u}{D^2} \,\, , \quad \dot{y}=-\frac{y^4}{D^2} \,\, , \quad \dot{u}=-2\frac{y^3u}{D^2} \,\, .
\end{equation} 

\begin{prop}\label{prop_K1} Let $g_0$ be a locally homogeneous Hermitian metric on a primary Kodaira surface $X$. Then, the solution $g(t)$ to the HCF starting from $g_0$ exists for all $t\geq0$. Moreover, 
$$
\left(X,(1{+}t)^{-1}g(t)\right) \GHa \text{\rm \{point\}} \quad \text{ as } t \to \infty \,\,.
$$
\end{prop}

\begin{proof}
Let $T\in(0,+\infty]$ denote the maximal existence time of the flow. Then, for any $t \in [0,T)$, it holds that 
\begin{equation}\label{estDuy*}
\begin{aligned} 
\dot{D}=\frac{y^3}{D}>0 \quad&\Longrightarrow\quad D(t)\geq D_0 \,\, , \\ 
\dot{y}<0 \,\,,\quad \dot{u}<0 \quad&\Longrightarrow\quad y(t)\leq y_0 \,\,,\quad u(t)\leq u_0 
\end{aligned}
\end{equation} 
and, on the other hand
\begin{equation} \label{estx*} \begin{aligned}
\dot{D} \leq \frac{y_0^3}{D} \quad&\Longrightarrow\quad D(t)\leq \sqrt{2ty_0^3+D_0^2} \,\,, \\
\dot{x} \leq \frac{2y_0^2}{D_0} \quad&\Longrightarrow\quad x(t) \leq \Big(\frac{2y_0^2}{D_0}\Big)t+x_0 \,\,.
\end{aligned} \end{equation} 
Therefore, the long-time existence of the solution follows from \eqref{estDuy*} and \eqref{estx*}. For the second claim, we notice that
\begin{equation}\label{estdotyu*} 
\begin{aligned} 
\lim_{t\rightarrow+\infty}\dot{y}(t)=0 \quad&\Longrightarrow\quad \lim_{t\rightarrow+\infty}\frac{y^2}{D}=0 \,\, , \\
\lim_{t\rightarrow+\infty}\dot{u}(t)=0 \quad&\Longrightarrow\quad \lim_{t\rightarrow+\infty}\frac{y^3u}{D^2}=0 \,\, .
\end{aligned}\end{equation}
Now, let us suppose by contradiction that $\frac{y^2u}{D^2} \rightarrow \delta>0$, as $t\rightarrow+\infty$. From this and \eqref{estdotyu*} it comes that $$
\dot{x} \sim -\frac{y^2u}{D^2} \quad \text{ as } t \to \infty
$$ 
and hence there exist $0<\delta'<\delta$ and $t_*>0$ such that, for any $t \in [t_*,+\infty)$, it holds
$$
\dot{x}\leq-\delta' \quad \Longrightarrow\quad x(t) \leq -\delta't+x(t_*) \,\,,
$$
which is not possible. As a consequence, we have that $\dot{x}(t) \rightarrow 0$ as $t\rightarrow+\infty$. From this last claim, arguing again by contradiction, we also get $(1{+}t)^{-1}x(t) \rightarrow 0$ as $t\rightarrow+\infty$. 
\end{proof}

\subsection{Secondary Kodaira surfaces} \hfill \par

The HCF on $\big(\R \ltimes {\rm H}_3(\R),J\big)$ reduces to the ODEs system 
\begin{equation} \label{IIKod} 
\dot{x}=\frac{2y^2D-u(x^2+y^2)}{D^2} \,\, , \quad \dot{y}=-\frac{y^2+u^2}{D^2} \,\, , \quad \dot{u}=-2\frac{yu(x^2+y^2)}{D^2} \,\, .
\end{equation} 

\begin{prop}\label{prop_K2} Let $g_0$ be a locally homogeneous Hermitian metric on a secondary Kodaira surface $X$. Then, the solution $g(t)$ to the HCF starting from $g_0$ exists for all $t\geq0$. Moreover
$$
\left(X,(1{+}t)^{-1}g(t)\right)\GHa \text{\rm \{point\}} \quad \text{ as } t \to \infty \,\,.
$$
\end{prop}

\begin{proof} 
Let $T\in(0,+\infty]$ be the maximal existence time of the solution. Then, for any $t \in [0,T)$ it holds
\begin{equation*}\label{estDuy} \begin{aligned} 
\dot{D}=\frac{y^3+xu}{D}>0 \quad&\Longrightarrow\quad D(t)\geq D_0 \,\,, \\ 
\dot{y}<0 \,\,,\quad \dot{u}<0 \quad&\Longrightarrow\quad y(t)\leq y_0 \,\,,\quad u(t)\leq u_0 \,\, . 
\end{aligned} \end{equation*} 
Moreover, since
\begin{equation*} \label{estx} 
\dot{x} < \frac{2y^2}{D} \leq \frac{2y_0^2}{D_0} \quad\Longrightarrow\quad x(t)\leq \Big(\frac{2y_0^2}{D_0}\Big)t+x_0 \,\, ,
\end{equation*} it follows that $T=+\infty$. For the second claim, we firstly suppose by contradiction that $u(t)\rightarrow u_{\infty}>0$ as $t\rightarrow+\infty$. Thus, since
\begin{equation} \label{estdotyu}\begin{aligned} 
\lim_{t\rightarrow+\infty}\dot{y}(t)&=0 \quad\Longrightarrow\quad \lim_{t\rightarrow+\infty}\frac{y}{D}=\lim_{t\rightarrow+\infty}\frac{u}{D}=0 \,\,, \\ 
\lim_{t\rightarrow+\infty}\dot{u}(t)&=0 \quad\Longrightarrow\quad \lim_{t\rightarrow+\infty}\frac{x^2yu}{D^2}=0 \,\,,
\end{aligned}  \end{equation} 
we have
\begin{equation*} 
0 \leq \frac{u_{\infty}}{D} \leq \frac{u}{D} \rightarrow 0 \quad\Longrightarrow\quad \lim_{t\rightarrow+\infty}D(t)=+\infty \quad\Longrightarrow\quad \lim_{t\rightarrow+\infty}x(t)y(t)=+\infty \,\,. \end{equation*} 
On the other hand, it follows by \eqref{estdotyu} that
$$
\frac{x^2yu}{D^2}=\frac1{1-\frac{u}{xy}}\cdot u \cdot \frac1{y-\frac{u}{x}} \rightarrow 0 \quad\Longrightarrow\quad y-\frac{u}{x} \rightarrow +\infty
$$ 
which is not possible, and hence $u(t)\rightarrow 0$ as $t\rightarrow+\infty$. 

Finally, let us assume by contradiction that $\frac{x^2u}{D} \rightarrow \delta>0$ as $t\rightarrow+\infty$. Then we get
$$
\dot{x} \sim -\frac{x^2u}{D^2} \quad \text{ as } t \to \infty
$$ and so there exist $0<\delta'<\delta$ and $t_*>0$ such that, for any $t \in [t_*,+\infty)$, we have 
$$
\dot{x}<-\delta' \quad\Longrightarrow\quad x(t) \leq -\delta't+x(t_*) \,\, ,
$$ 
which is absurd. Consequently it comes $\dot{x}(t) \rightarrow 0$ as $t\rightarrow+\infty$. Arguing again by contradiction, we finally get $(1{+}t)^{-1}x(t) \rightarrow 0$ as $t\rightarrow+\infty$. 
\end{proof}

\subsection{Inoue surfaces of type $S^0$} \hfill \par

The HCF on $\big({\rm Sol}_0^4,J_{a,b}\big)$ reduces to the ODEs system
\begin{equation}\left.\begin{aligned}
\dot{x} &= -(9a^2+b^2)\frac{x^2u}{D^2} \\
\dot{y} &= 8a^2-(9a^2+b^2)\Big(\frac{u}{D}\Big)^2 \\
\dot{u} &= -2(9a^2+b^2)\frac{x^2u}{D^2}y
\end{aligned}\right.. \end{equation}

\begin{prop}\label{prop_S0}\label{Inoue0} Let $g_0$ be a locally homogeneous Hermitian metric on an Inoue surfaces $X$ of type $S^0$. Then, the solution $g(t)$ to the HCF starting from $g_0$ exists for all $t\geq0$. In particular, $y(t)\sim 8a^2\,t$ and $x(t)<x_0$, $u(t)<u_0$ for any $t>0$.
\end{prop}

\begin{proof} 
Let $T\in(0,+\infty]$ denotes the maximal existence time of the solution. For any $t \in [0,T)$ we have
\begin{equation*} \begin{aligned} 
\dot{D}=8a^2x+(9a^2+b^2)\frac{xu}{D}>0 \,\, , \quad&\Longrightarrow\quad D(t)\geq D_0 \,\,, \\ 
\dot{x}<0 \,\,,\quad \dot{u}<0 \quad&\Longrightarrow\quad x(t)\leq x_0 \,\,,\quad u(t)\leq u_0 \,\,.
\end{aligned} \end{equation*} 
Moreover, since
\begin{equation*}
\dot y\leq\frac {8a^2xy}{D}<\frac {8a^2x_0y}{D_0}\quad \Longrightarrow\quad y< y_0e^{kt}\,\,,
\end{equation*} 
where $k:=\frac {8a^2x_0}{x_0y_0-|z_0|^2}$, it follows that $T=+\infty$.

For the second claim, let us assume by contradiction that $\frac{u}{D} \rightarrow \delta>0$, i.e. $u \rightarrow u_{\infty}>0$ and $D \rightarrow D_{\infty}<\infty$. Then, there exists a finite time $t_*>0$ and a constant $k_1>1$ such that, for any $t\geq t_*$,
$$
-k_1 x(t)^2 \leq \dot{x}(t) \leq -\tfrac1{k_1} x(t)^2
$$ 
and hence 
\begin{equation} \frac1{k_1(t-t_*)+\frac1{x(t_*)}} \leq x(t) \leq \frac1{\frac1{k_1}(t-t_*)+\frac1{x(t_*)}}
\label{x(t)abs} 
\end{equation} 
Up to enlarge $t_*$, we can also assume that there exists $k_2 >1$ such that
$$
-k_2 x(t) \leq \dot{u}(t) \leq -\tfrac1{k_2} x(t) \quad \text{ for any $t\geq t_*$} 
$$ 
and so, by means of \eqref{x(t)abs} 
$$
-k_2\frac1{\frac1{k_1}(t-t_*)+\frac1{x(t_*)}} \leq \dot{u}(t) \leq -\tfrac1{k_2}\frac1{k_1(t-t_*)+\frac1{x(t_*)}}\,\,, 
$$ 
for any $t\geq t_*$. This leads us to 
$$
u(t_*)-k_1k_2\log\big(\tfrac{x(t_*)}{k_1}(t-t_*)+1\big) \leq u(t) \leq u(t_*)-\tfrac1{k_1k_2}\log\big(k_1x(t_*)(t-t_*)+1\big)\,\,, 
$$ 
for any $t\geq t_*$, and hence $\lim_{t \to +\infty}u(t)=-\infty$, which is not possible. Therefore, $\frac{u}{D} \rightarrow 0$ must hold and we have 
$$
\dot{y}(t) \rightarrow 8a^2
$$ 
as $t \rightarrow +\infty$.
\end{proof}

Then, in view of this result, we have 

\begin{prop}\label{GH_inoueS0} Let $X$ be an Inoue surface of type $S^0$ and $g(t)$ be a locally homogeneous solution to the HCF on $X$. Then
$$
\left(X,(1{+}t)^{-1}g(t)\right)\GHa S^1\big(\tfrac{\sqrt2a}{\pi}\big) \quad \text{as}\,\,\, t\to \infty \,\,,
$$ 
where $S^1\big(\tfrac{\sqrt2a}{\pi}\big)=\big\{z \in \C : |z|=\frac{\sqrt2a}{\pi}\big\}$ is the circle of length $2\sqrt2a$. \end{prop}

In order to prove this statement, we recall the underlying geometry of the Inoue surfaces of type $S^0$. Let $a,b\in\R$, with $a>0$ and $b\neq0$, and $A \in {\rm SL}(3,\Z)$ be a matrix with eigenvalues 
$$
e^{2\sqrt2a}\,\,,\quad e^{\sqrt2(-a+\sqrt{-1}b)}\,\,,\quad e^{\sqrt2(-a-\sqrt{-1}b)}\,\,.
$$
The pair $G_{a,b}:=\big({\rm Sol}_0^4,J_{a,b}\big)$ can be realized as the group of complex $3\times3$ matrices of the form 
$$
G_{a,b} = \left\{M(p,q,r,s):=\left({\scalefont{0.8} 
\begin{array}{ccc} \!e^{s\sqrt2(-a+\sqrt{-1}b)}\! & \!0\! & \!p+\sqrt{-1}q\! \\ \!0\! & \!e^{s2\sqrt2a}\! & \!r\! \\ \!0\! & \!0\! & \!1\! 
\end{array}}\right) \, : \quad p, q, r, s \in \R \right\} \,\, .
$$ 
Indeed, let $\{E^i_j\}$ denote the standard basis of $\g\l(3,\C)$. Then the Lie algebra $\g_{a,b}:={\rm Lie}(G_{a,b}) \subset \g\l(3,\C)$ is the $\R$-span of 
$$
X_1:=(1-\sqrt{-1})E^1_3 \,\, , \quad X_2:=(1+\sqrt{-1})E^1_3 \,\, , \quad X_3:=E^2_3 \,\, , \quad X_4:=\sqrt2(-a+\sqrt{-1}b)E^1_1+2\sqrt2a\,E^2_2 \,\,.
$$ 
Since the structure constants of $\g_{a,b}$ with respect to $\{X_i\}$ are given by
\begin{equation*} [X_1,X_4]=\sqrt2aX_1-\sqrt2bX_2 \,\, , \quad [X_2,X_4]=\sqrt2bX_1+\sqrt2aX_2 \,\, , \quad [X_3,X_4]=-2\sqrt2aX_3 \,\,,
\label{bracketS0R} \end{equation*}
setting 
\begin{equation*} 
Z_1:=\frac{X_1-\sqrt{-1}X_2}{\sqrt2} \,\, , \quad Z_2:=\frac{X_3-\sqrt{-1}X_4}{\sqrt2} \,\, , \label{frameRC} 
\end{equation*} 
one obtains the structure constants given in Section \ref{model_sub}. Let now $(v_1,v_2,v_3)^t\in \R^3$ and $(w_1,w_2,w_3)^t \in \C^3$ be the eigenvectors of $e^{2\sqrt2a}$ and $e^{\sqrt2(-a+\sqrt{-1}b)}$, respectively, and consider the lattice $\Gamma_{a,b} \subset G_{a,b}$ generated by 
$$
h_0:=\left({\scalefont{0.8} \begin{array}{ccc} \!e^{\sqrt2(-a+\sqrt{-1}b)}\! & \!0\! & \!0\! \\ \!0\! & \!e^{2\sqrt2a}\! & \!0\! \\ \!0\! & \!0\! & \!1\! \end{array}}\right) \,\, , \quad h_i:=\left({\scalefont{0.8} \begin{array}{ccc} \!1\! & \!0\! & \!w_i\! \\ \!0\! & \!1\! & \!v_i\! \\ \!0\! & \!0\! & \!1\! \end{array}}\right) \,\, , \quad i=1,2,3 \,\, .
$$ 
Then the left action of $\Gamma_{a,b}$ on $G_{a,b}$ is explicitly given by
\begin{equation} \begin{aligned}
h_0 \cdot M(p,q,r,s) &= M\big(e^{-\sqrt2a}(\cos(\sqrt2b)p-\sin(\sqrt2b)q),e^{-\sqrt2a}(\sin(\sqrt2b)p+\cos(\sqrt2b)q),e^{2\sqrt2a}r,s+1\big) \\
h_i \cdot M(p,q,r,s) &= M(p+{\rm Re}(w_i),q+{\rm Im}(w_i),r+v_i,s)
\end{aligned} \label{actionGamma} \end{equation} 
and the quotient $X=\Gamma_{a,b}\backslash G_{a,b}$ is an Inoue surface of type $S^0$.

\begin{proof}[Proof of Proposition \ref{GH_inoueS0}] Let $X=\Gamma_{a,b}\backslash G_{a,b}$ be an Inoue surface of type $S^0$ and $g(t)$ a locally homogeneous solution to the HCF on $X$. By \eqref{actionGamma}, the projection 
$$
G_{a,b} \to \R \,\, , \quad M(p,q,r,s) \mapsto s
$$ 
factorizes to a map $\pi: X \to S^1=\R/\Z$, which is a fibration with standard fiber $T^3$ (see \cite{Ino}). On the other hand, the path 
$$
\R \to G_{a,b} \,\, , \quad s \mapsto M(0,0,0,s)
$$ 
factorizes to a section $\gamma: S^1=\R/\Z \to X$ whose length with respect to $g(t)$ is 
\begin{equation}
\ell_{g(t)}(\gamma)=\sqrt{y(t)} \,\, .
\label{lengthcircle} \end{equation} 
Notice also that by Proposition \ref{Inoue0} 
$$
(1{+}t)^{-1}g(t) \to \tilde{g}_{\infty} := \left({\scalefont{0.8} \begin{array}{cc} \!0\! & \!0\! \\ \!0\! & \!8a^2\! \end{array}}\right) \quad \text{ as $t \to \infty$ }\,\,.
$$ 
Moreover, in analogy with \cite[Lemma 5.2]{TW}, the kernel of $\tilde{g}_{\infty}$ is the integrable distribution $\mathcal D$ spanned by $X_1,X _2$, which is dense inside any fiber of $\pi$. Finally, the claim follows by \eqref{lengthcircle} and this last observation (see e.g. \cite[Cor 3.18]{Bol}).
\end{proof}

\subsection{Inoue surfaces of type $S^\pm$} \hfill \par

The HCF on $\big({\rm Sol}_1^4,J_{1}\big)$ reduces to the ODEs system \begin{equation} \label{InoueS+}
\begin{aligned}
\dot{x}= 3-\frac {u|z-\bar{z}|^2}{D^2} \,\, , \quad \dot{y}&= -\frac{y^2|z-\bar{z}|^2}{D^2} \,\, , \quad \dot{u}= -\frac{2xy^2|z-\bar{z}|^2}{D^2}\,.
\end{aligned} \end{equation} 

\begin{prop}\label{Inoue+} Let $g_0$ be a locally homogeneous Hermitian metric on an Inoue surfaces $X$ of type $S^\pm$ obtained by $\big({\rm Sol}_1^4,J_{1}\big)$. Then, the solution $g(t)$ to the HCF starting from $g_0$ exists for all $t\geq0$. In particular, $x(t)\sim 3t$ and $y(t)<y_0$, $u(t)<u_0$ for any $t>0$.
\end{prop}

\begin{proof} Let $T\in(0,+\infty]$ be the maximal existence time of the flow. Then, for any $t \in [0,T)$, we have
\begin{equation}\label{decr_S+} 
\begin{aligned} 
\dot{D}=3y+\frac{y|z-\bar z|^2}{D}\geq 0 \,\,,& \\ 
\dot{y}<0 \,\,, \quad \dot{u}<0 \quad\Longrightarrow\quad y(t)\leq y_0 \,\,,& \quad u(t)\leq u_0 \,\,. 
\end{aligned}
\end{equation}
On the other hand
$$
\dot{x}= 3-\frac {u|z-\bar z|^2}{D^2}\leq 3\quad \Longrightarrow\quad x(t)\leq 3t+x_0
$$
and the long-time existence follows, i.e. $T=+\infty$. Finally, to conclude the proof it is enough to show 
\begin{equation}\label{lim-rap}
\lim_{t\to \infty} \frac {|z-\bar z|}{D} =0\,\,.
\end{equation}
Let us assume by contradiction that $\frac{|z-\bar{z}|}{D} \rightarrow \epsilon>0$. Then, by the means of \eqref{InoueS+} and \eqref{decr_S+}, there exists $t_*>0$ and a constant $k_1>1$ such that 
$$
-k_1 y(t)^2 \leq \dot{y}(t) \leq -\frac1{k_1} y(t)^2 \quad \text{ for any $t\geq t_*$ } \,\,.
$$ 
This in turn implies, for any $t\geq t_\ast$,
\begin{equation} 
\frac1{k_1(t-t_*)+\frac1{y(t_*)}} \leq y(t) \leq \frac1{\frac1{k_1}(t-t_*)+\frac1{y(t_*)}}\,\, . \label{y(t)abs} 
\end{equation} 
Besides, up to enlarge $t_*$, there also exists a constat $k_2 >1$ such that 
$$
-k_2 y(t) \leq \dot{u}(t) \leq -\tfrac1{k_2} y(t) \quad \text{ for any $t\geq t_*$} \,\,.
$$ 
Therefore, since \eqref{y(t)abs} holds, for any $t\geq t_\ast$ we have
$$
-k_2\frac1{\frac1{k_1}(t-t_*)+\frac1{y(t_*)}} \leq \dot{u}(t) \leq -\frac1{k_2}\frac1{k_1(t-t_*)+\frac1{y(t_*)}}
$$
and 
$$
u(t_*)-k_1k_2\log\left(\frac{y(t_*)}{k_1}(t-t_*)+1\right) \leq u(t) \leq u(t_*)-\frac1{k_1k_2}\log\left(k_1y(t_*)(t-t_*)+1\right)\,\,.
$$ 
Nonetheless, this would imply $\lim_{t\rightarrow +\infty}u(t)=-\infty$, which is not possible. Hence, \eqref{lim-rap} holds and $x\sim3t$ follows.
\end{proof}

The HCF on $\big({\rm Sol}_1^4,J_{2}\big)$ reduces to the ODEs system 
\begin{equation} \label{InoueS-}
\left.\begin{aligned}
\dot x &=3+\frac{2y^2}{D}-\frac{u(y^2+|z-\bar{z}|^2)}{D^2} \\
\dot y &=-\frac{y^2(|z-\bar{z}|^2+y^2)}{D^2}\\
\dot u&=-\frac{2y^2(x|z-\bar{z}|^2+yu)}{D^2}
\end{aligned}\right.. \end{equation}

\begin{prop}\label{Inoue-} Let $g_0$ be a locally homogeneous Hermitian metric on an Inoue surfaces $X$ of type $S^+$ obtained by $\big({\rm Sol}_1^4,J_{2}\big)$. Then, the solution $g(t)$ to the HCF starting from $g_0$ exists for all $t\geq0$. In particular, $x(t)\sim 3t$ and $y(t)<y_0$, $u(t)<u_0$ for any $t>0$.
\end{prop}

\begin{proof}
Let $T\in(0,+\infty]$ denote the maximal existence time of the solution. Then, a direct computation yields that
\begin{equation}\label{decr_S-} 
\begin{aligned} 
&\dot D = 3y+\frac{y(|z-\bar z|^2+y^2)}{D^2}\geq 0  \,\,, \\ 
\dot{y}<0 \,\,,& \quad \dot{u}<0 \quad\Longrightarrow\quad y(t)\leq y_0 \,\,, \quad u(t)\leq u_0 \,\,. 
\end{aligned}\end{equation}
On the other hand, since
$$
\dot x \leq 3+\frac{2y^2}{D} \leq 3+\frac{2y_0^2}{D_0}\,\,,
$$
we have $T=+\infty$ and the first part of the claim follows. To conclude the proof it is enough to show that
\begin{equation} \label{AAA}
\lim_{t\to \infty}\frac{2y^2}{D}=\lim_{t\to \infty}\frac{u(y^2+|z-\bar{z}|^2)}{D^2}=0 \,\, .
\end{equation}
By the means of \eqref{decr_S-}, we can have either 
$$
\lim_{t\rightarrow+\infty}D(t)=+\infty \qquad \text{or}\qquad \lim_{t\rightarrow+\infty}D(t)<+\infty\,\, .
$$
In the former case, \eqref{decr_S-} directly implies \eqref{AAA}. Let us assume then $\lim_{t\rightarrow+\infty}D(t)<+\infty$. By means of \eqref{decr_S-}, this implies $y(t)\to 0$ for $t\to\infty$, and hence $\lim_{t\to \infty}\frac{2y^2}{D}=0$. Moreover, using the same argument as in the proof of Proposition \ref{Inoue+}, one can prove that $\frac{|z-\bar{z}|}{D}\to0$ necessarily holds, and so we obtain \eqref{AAA}.
\end{proof}

In view of the above results, we have

\begin{prop}\label{GH_inoS+-} Let $X$ be an Inoue surface of type $S^\pm$ and $g(t)$ be a locally homogeneous solution to the HCF on $X$. Then
$$
\left(X,(1{+}t)^{-1}g(t)\right) \GHa S^1\big(\tfrac{\sqrt{3}}{2\pi}\big) \quad \text{as}\,\,\, t\to \infty \,\,,
$$ 
where $S^1\big(\tfrac{\sqrt{3}}{2\pi}\big)=\{z \in \C : |z|=\tfrac{\sqrt{3}}{2\pi} \}$ is the circle of length $\sqrt{3}$.
\end{prop}

We briefly recall the construction of Inoue surfaces of type $S^+$. Let $N\in{\rm SL}(2,\Z)$ be a unimodular matrix with real positive eigenvalues given by $\lambda$ and $\lambda^{-1}$, with $\lambda>1$. It is well known that any $S^+$ surface can be realized as the quotient of the group
$$
G_+:=\left\{M_+(r,q,v,u):=\left({\scalefont{0.8} 
\begin{array}{ccc} \!1\! & \!u\! & \!v\! \\ \!0\! & \!q\! & \!r\! \\ \!0\! & \!0\! & \!1\! 
\end{array}}\right) \, : \quad r,v,u \in \R,\quad q\in\R_{>0} \right\} \,\, .
$$
by a lattice $\Gamma_+:=\langle f_0,f_1,f_2,f_3\rangle$, where $f_i\in G_+$ are defined starting from $N$ (see \cite{Ino}).

Notice that Inoue surfaces of type $S^\pm$ enjoy nearly the same properties of surfaces of type $S^0$ (see \cite{Ino}). In particular, they do not contain complex curves and any $S^+$ surface is diffeomorphic to a bundle over $S^1$. Moreover, since any $S^-$ surface admits an unramified double cover given by a $S^+$ surface, it is enough to prove the statement for Inoue surfaces of type $S^+$.

\begin{proof}[Proof of Proposition \ref{GH_inoS+-}] Let $X=\Gamma_+\backslash G_+$ be an Inoue surface of type $S^+$ and $g(t)$ a locally homogeneous solution to the HCF on $X$. The application
$$
G_+ \to \R \,\, , \quad M_+(r,q,v,u) \mapsto \tfrac{\log q}{\log \lambda}
$$ 
factorizes to a map $\pi: X \to S^1$, which is a locally trivial fibration (see \cite{Ino}). On the other hand, the path
$$
\R \to G_{+} \,\, , \quad s \mapsto M_+(0,\lambda^s,0,0)
$$ 
factorizes to a section $\gamma: S^1 \to X$ whose length with respect to $g(t)$ is 
\begin{equation*}
\ell_{g(t)}(\gamma)=\sqrt{x(t)} \,\, .
\label{lengthcircle-} \end{equation*} 
Now, in view of the above results
$$
(1{+}t)^{-1}g(t) \to \tilde{g}_{\infty} := \left({\scalefont{0.8} \begin{array}{cc} \!3\! & \!0\! \\ \!0\! & \!0\! \end{array}}\right) \quad \text{ as $t \to \infty$ }\,\,.
$$ 
Again, the kernel of $\tilde{g}_{\infty}$ is the integrable distribution $\mathcal D$ spanned by the real and imaginary part of $Z_2$, which is dense inside any fiber of $\pi$ (see \cite[Lemma 6.2]{TW}). Therefore, in analogy with the case of $S^0$ surfaces, the claim follows.
\end{proof}

We are now in position to prove Theorem \ref{theor_A} and Theorem \ref{theor_B}.

\begin{proof}[Proof of Theorem \ref{theor_A} and Theorem \ref{theor_B}] Let $X$ be a compact complex surface and $g_0$ a locally homogeneous non-K\"ahler metric on $X$. By Theorem \ref{theor-Wall} and Remark \ref{remKE} $X$ is a quotient $\Gamma\backslash G$, where  $G$ is one of the Lie groups listed in Section \ref{model_sub}, i.e.
$$
\widetilde{\rm SE}(2)\times\R\,\,,\quad {\rm SU(2)}\times\R\,\,, \quad \widetilde{\rm SL}(2,\R)\times\R\,\,, \quad \R \times {\rm H}_3(\R)\,\,, \quad \R \ltimes {\rm H}_3(\R)\,\,, \quad {\rm Sol}_0^4\,\,,\quad {\rm Sol}_1^4\,\, ,
$$
and $\Gamma\subset G$ is a co-compact lattice.

Let also $T\in(0,+\infty]$ be the extinction time of the HCF solution starting from $g_0$. Then, by means of Proposition \ref{prop_hyper}, Proposition \ref{prop_hopf}, Proposition \ref{prop_ell}, Proposition \ref{prop_K1}, Proposition \ref{prop_K2}, Proposition \ref{prop_S0}, Proposition \ref{Inoue+} and Proposition \ref{Inoue-}, we have ${T<\infty}$ if and only $G{={\rm SU(2)}\times\R}$. This implies Theorem \ref{theor_A}.

Finally, Theorem \ref{theor_B} comes from Proposition \ref{prop_hyper}, Proposition \ref{prop_ell_GH}, Proposition \ref{prop_K1}, Proposition \ref{prop_K2}, Proposition \ref{GH_inoueS0} and Proposition \ref{GH_inoS+-}.
\end{proof}

\section{Appendix}

In this Appendix, we explicitly write down the tensors $S$ and $Q^i$ used in Section \ref{model_sub} to obtain the HCF tensor $K$. We assume $G$ to be one of the (non-abelian) Lie groups listed in Section \ref{model_sub}, equipped with a left-invariant Hermitian structure $(J,g)$ as in \eqref{leftinvmetric}. Our results directly follow by the formulas given in Section \ref{prel} and the structure equations of $G$.

\subsubsection*{Hyperelliptic surfaces} \hfill \par

$$
\begin{aligned}
S_{1\bar 1}&=\tfrac{x^2|z|^2}{(xy-|z|^2)^2}\,\, ,		\qquad &S_{2\bar 2}&=\tfrac{xy|z|^2}{(xy-|z|^2)^2}\,\,,		\qquad &S_{1\bar 2}&=\tfrac{x^2yz}{(xy-|z|^2)^2}\,\, ,\\
Q^1_{1\bar 1}&=\tfrac{x^2|z|^2}{(xy-|z|^2)^2}\,\, ,	\qquad &Q^1_{2\bar 2}&=\tfrac{xy|z|^2}{(xy-|z|^2)^2}\,\,,	\qquad &Q^1_{1\bar 2}&=\tfrac{xz|z|^2}{(xy-|z|^2)^2}\,\, ,\\
Q^2_{1\bar 1}&=0\,\, ,						\qquad &Q^2_{2\bar 2}&=\tfrac{2|z|^2}{xy-|z|^2}\,\,,		\qquad &Q^2_{1\bar 2}&=0\,\, ,\\
Q^3_{1\bar 1}&=\tfrac{x^2|z|^2}{(xy-|z|^2)^2}\,\, ,	\qquad &Q^3_{2\bar 2}&=\tfrac{|z|^4}{(xy-|z|^2)^2}\,\,,	\qquad &Q^3_{1\bar 2}&=\tfrac{xz|z|^2}{(xy-|z|^2)^2}\,\, ,\\
Q^4_{1\bar 1}&=0\,\, ,						\qquad &Q^4_{2\bar 2}&=\tfrac{|z|^2}{xy-|z|^2}\,\,,		\qquad &Q^4_{1\bar 2}&=0\,\, .
\end{aligned}
$$

\subsubsection*{Hopf surfaces} \hfill \par

$$\begin{aligned}
S_{1\bar 1}&=\tfrac{x (x^3(1+\lambda^2)+|z|^2(2x+y) )}{(xy-|z|^2)^2}\,\,,\\
%$$ 
%$$
S_{2\bar 2}&=\tfrac{-(1+\lambda^2)x^2(xy-2|z|^2)-4|z|^2(xy-|z|^2)+y^2(2x^2+|z|^2)}{(xy-|z|^2)^2}\,\,,\\
%$$
%$$
S_{1\bar 2}&=\tfrac{xz (-i\lambda(xy-|z|^2)+x^2(1+\lambda^2)+y(x+y)+|z|^2 )}{(xy-|z|^2)^2}\,\,,\\
%$$
%$$
Q^1_{1\bar 1}&=\tfrac{x ((1+\lambda^2)x^3(xy-|z|^2)^2+y|z|^2(x^2y^2-|z|^4)+(x+y)(xy-2|z|^2)|z|^4+2x^2y(xy-|z|^2) )}{(xy-|z|^2)^4}\,\,,\\
%$$
%$$
Q^1_{2\bar 2}&=\tfrac{y ((1+\lambda^2)x^3(xy-|z|^2)^2+y|z|^2(x^2y^2-|z|^4)+(x+y)(xy-2|z|^2)|z|^4+2x^2y(xy-|z|^2) )}{(xy-|z|^2)^4}\,\,,\\
%$$
%$$
Q^1_{1\bar 2}&=\tfrac{z ((1+\lambda^2)x^3+(2x+y)|z|^2 )}{(xy-|z|^2)^2}\,\,,\\
%$$
%$$
Q^2_{1\bar 1}&=\tfrac{2|z|^2}{xy-|z|^2}\,\,,\\
%$$ 
%$$
Q^2_{2\bar 2}&=\tfrac{2(1+\lambda^2)x^2}{xy-|z|^2}\,\,,\\
%$$ 
%$$
Q^2_{1\bar 2}&=\tfrac{-2xy(1+i\lambda)}{xy-|z|^2}\,\,,\\
Q^3_{1\bar 1}&=\tfrac{(1+\lambda^2)x^4+(2x^2+|z|^2)|z|^2}{(xy-|z|^2)^2}\,\,,\\
%$$
%$$
Q^3_{2\bar 2}&=\tfrac{(1+\lambda^2)x^2+(2x+y)y|z|^2}{(xy-|z|^2)^2}\,\,,\\
%$$
%$$
Q^3_{1\bar 2}&=\tfrac{z ((1+\lambda^2)x^3+i\lambda x+(x+y)|z|^2+x^2y )}{(xy-|z|^2)^2}\,\,,\\
%$$
%$$
Q^4_{1\bar 1}&=\tfrac{|z|^2}{xy-|z|^2}\,\,,\\
%$$
%$$
Q^4_{2\bar 2}&=\tfrac{(1+\lambda^2)x^2}{xy-|z|^2} \,\,,\\
%$$
%$$
Q^4_{1\bar 2}&=\tfrac{-(1+i\lambda)xz}{xy-|z|^2} \,\,.
\end{aligned}$$
Here, $\lambda\in\R$ denotes the parameter of the family of complex structures related to Hopf surfaces.

%
%
%$$
%\begin{aligned}
%Q^2_{1\bar 1}&=\tfrac{2|z|^2}{xy-|z|^2}\,\, ,			\qquad &Q^2_{2\bar 2}&=\tfrac{2(1+\lambda^2)x^2}{xy-|z|^2}\,\,,\qquad &Q^2_{1\bar 2}&=\tfrac{-2xy(1+i\lambda)}{xy-|z|^2}\,\, ,\\
%Q^3_{1\bar 1}&=\tfrac{(1+\lambda^2)x^4+(2x^2+|z|^2)|z|^2}{(xy-|z|^2)^2}\,\, ,	\qquad &Q^3_{2\bar 2}&=\tfrac{(1+\lambda^2)x^2+(2x+y)y|z|^2}{(xy-|z|^2)^2}\,\,,	\qquad &Q^3_{1\bar 2}&=\tfrac{z((1+\lambda^2)x^3+i\lambda x+(x+y)|z|^2+x^2y)}{(xy-|z|^2)^2}\,\, ,\\
%Q^4_{1\bar 1}&=\tfrac{|z|^2}{xy-|z|^2}\,\, ,						\qquad &Q^2_{2\bar 2}&=\tfrac{(1+\lambda^2)x^2}{xy-|z|^2} \,\,,		\qquad &Q^2_{1\bar 2}&=\tfrac{-(1+i\lambda)xz}{xy-|z|^2} \,\, .
%\end{aligned}
%$$

\subsubsection*{Non-K\"ahler properly elliptic surfaces} \hfill \par

$$\begin{aligned}
S_{1\bar 1} &= \tfrac{-y (2x+(1+\lambda^2)y )(xy-|z|^2)+ ((x+y)^2+a^2y^2-4 )|z|^2}{(xy-|z|^2)^2} \,\,,\\
S_{2\bar 2} &= \tfrac{y ((1+\lambda^2)y^3+(x-2y)|z|^2 )}{(xy-|z|^2)^2}\,\,, \\
S_{1\bar 2} &= \tfrac{yz ((1+i\lambda)(xy-|z|^2)+x^2-2xy+(1+\lambda^2)y^2 )}{(xy-|z|^2)^2} \,\,,\\
Q^1_{1\bar 1} &= \tfrac{x ((1+\lambda^2)y^3+(x-2y)|z|^2 )}{(xy-|z|^2)^2}\,\,, \\
Q^1_{2\bar 2} &= \tfrac{y ((1+\lambda^2)y^3+(x-2y)|z|^2 )}{(xy-|z|^2)^2}\,\,, \\
Q^1_{1\bar 2} &= \tfrac{z ((1+\lambda^2)y^3+(x-2y)|z|^2 )}{(xy-|z|^2)^2}\,\,, \end{aligned}$$
$$\begin{aligned}
Q^2_{1\bar 1} &= \tfrac{2y^2(1+\lambda^2)}{xy-|z|^2}\,\,, \\
Q^2_{2\bar 2} &= \tfrac{2|z|^2}{xy-|z|^2}\,\,,\\
Q^2_{1\bar 2} &= \tfrac{2(1+i\lambda)yz}{xy-|z|^2}\,\,, \\
Q^3_{1\bar 1} &= \tfrac{ ((1+\lambda^2)y^2+x(x-2y) )|z|^2}{(xy-|z|^2)^2}\,\,, \\
Q^3_{2\bar 2} &= \tfrac{(1+\lambda^2)y^4+|z|^2(|z|^2 -2y^2)}{(xy-|z|^2)^2}\,\,,\\
Q^3_{1\bar 2} &= \tfrac{z ((1-i\lambda)y(xy-|z|^2) +(1+\lambda^2)y^3 +x|z|^2 )}{(xy-|z|^2)^2}\,\,,\\
Q^4_{1\bar 1} &= \tfrac{y^2(1+\lambda^2)}{xy-|z|^2}\,\,, \\
Q^4_{2\bar 2} &= \tfrac{|z|^2}{xy-|z|^2}\,\,, \\
Q^4_{1\bar 2} &= \tfrac{1+i\lambda}{xy-|z|^2}\,\,.\\
\end{aligned}$$
Here, $\lambda\in\R$ denotes the parameter of the family of complex structures related to non-K\"ahler properly elliptic surfaces.

\subsubsection*{Primary Kodaira surfaces} \hfill \par
$$
\begin{aligned}
S_{1\bar 1}&=\tfrac{-y^2(xy-2|z|^2)}{(xy-|z|^2)^2}\,\, ,		\qquad &S_{2\bar 2}&=\tfrac{y^4}{(xy-|z|^2)^2}\,\,,		\qquad &S_{1\bar 2}&=\tfrac{y^3z}{(xy-|z|^2)^2}\,\, ,\\
Q^1_{1\bar 1}&=\tfrac{xy^3}{(xy-|z|^2)^2}\,\, ,	\qquad &Q^1_{2\bar 2}&=\tfrac{y^4}{(xy-|z|^2)^2}\,\,,	\qquad &Q^1_{1\bar 2}&=\tfrac{y^3z}{(xy-|z|^2)^2} \,\, ,\\
Q^2_{1\bar 1}&= \tfrac{2y^2}{xy-|z|^2}\,\, ,			\qquad &Q^2_{2\bar 2}&=0\,\,,		\qquad &Q^2_{1\bar 2}&=0\,\, ,\\
Q^3_{1\bar 1}&=\tfrac{y^2|z|^2}{(xy-|z|^2)^2} \,\, ,	\qquad &Q^3_{2\bar 2}&=\tfrac{y^4}{(xy-|z|^2)^2}\,\,,	\qquad &Q^3_{1\bar 2}&=\tfrac{y^3z}{(xy-|z|^2)^2}\,\, ,\\
Q^4_{1\bar 1}&=\tfrac{y^2}{xy-|z|^2}\,\, ,						\qquad &Q^4_{2\bar 2}&=0\,\,,		\qquad &Q^4_{1\bar 2}&=0\,\, .
\end{aligned}
$$
%
%$$\begin{aligned}
%S_{1\bar 1} &= \tfrac{-y^2(xy-2|z|^2)}{(xy-|z|^2)^2} \\
%S_{2\bar 2} &= \tfrac{y^4}{(xy-|z|^2)^2} \\
%S_{1\bar 2} &= \tfrac{y^3z}{(xy-|z|^2)^2} \\
%Q^1_{1\bar 1} &= \tfrac{xy^3}{(xy-|z|^2)^2} \\
%Q^1_{2\bar 2} &= \tfrac{y^4}{(xy-|z|^2)^2} \\
%Q^1_{1\bar 2} &= \tfrac{y^3z}{(xy-|z|^2)^2} \\
%Q^2_{1\bar 1} &= \tfrac{2y^2}{xy-|z|^2} \\
%Q^2_{2\bar 2} &= 0 \\
%Q^2_{1\bar 2} &= 0 \\
%Q^3_{1\bar 1} &= \tfrac{y^2|z|^2}{(xy-|z|^2)^2} \\
%Q^3_{2\bar 2} &= \tfrac{y^4}{(xy-|z|^2)^2} \\
%Q^3_{1\bar 2} &= \tfrac{y^3z}{(xy-|z|^2)^2} \\
%Q^4_{1\bar 1} &= \tfrac{2y^2}{xy-|z|^2} \\
%Q^4_{2\bar 2} &= 0 \\
%Q^4_{1\bar 2} &= 0 \\
%\end{aligned}$$

\subsubsection*{Secondary Kodaira surfaces} \hfill \par
$$
\begin{aligned}
S_{1\bar 1}&=\tfrac{(x^2+y^2)|z|^2-y^2(xy-|z|^2)}{(xy-|z|^2)^2}\,\, ,		\qquad &S_{2\bar 2}&=\tfrac{y(x|z|^2+y^3)}{(xy-|z|^2)^2}\,\,,		\qquad &S_{1\bar 2}&=\tfrac{(x^2+y^2)+i(xy-|z|^2)}{(xy-|z|^2)^2}\,\, ,\\
Q^1_{1\bar 1}&=\tfrac{x(x|z|^2+y^3)}{(xy-|z|^2)^2}\,\, ,	\qquad &Q^1_{2\bar 2}&=\tfrac{y(x|z|^2+y^3)}{(xy-|z|^2)^2}\,\,,	\qquad &Q^1_{1\bar 2}&=\tfrac{z(x|z|^2+y^3)}{(xy-|z|^2)^2} \,\, ,\\
Q^2_{1\bar 1}&= \tfrac{2y^2}{xy-|z|^2}\,\, ,			\qquad &Q^2_{2\bar 2}&=\tfrac{2|z|^2}{xy-|z|^2}\,\,,		\qquad &Q^2_{1\bar 2}&=\tfrac{2iyz}{xy-|z|^2}\,\, ,\\
Q^3_{1\bar 1}&=\tfrac{(x^2+y^2)|z|^2}{(xy-|z|^2)^2} \,\, ,	\qquad &Q^3_{2\bar 2}&=\tfrac{y^4+|z|^4}{(xy-|z|^2)^2}\,\,,	\qquad &Q^3_{1\bar 2}&=\tfrac{z(x+iy)(|z|^2-iy^2)}{(xy-|z|^2)^2}\,\, ,\\
Q^4_{1\bar 1}&=\tfrac{y^2}{xy-|z|^2}\,\, ,						\qquad &Q^4_{2\bar 2}&=\tfrac{|z|^2}{xy-|z|^2}\,\,,		\qquad &Q^4_{1\bar 2}&=\tfrac{iyz}{xy-|z|^2}\,\, .
\end{aligned}
$$
%$$\begin{aligned}
%S_{1\bar 1} &= \tfrac{(x^2+y^2)|z|^2-y^2(xy-|z|^2)}{(xy-|z|^2)^2} \\
%S_{2\bar 2} &= \tfrac{y(x|z|^2+y^3)}{(xy-|z|^2)^2} \\
%S_{1\bar 2} &= \tfrac{(x^2+y^2)+i(xy-|z|^2)}{(xy-|z|^2)^2} \\
%Q^1_{1\bar 1} &= \tfrac{x(x|z|^2+y^3)}{(xy-|z|^2)^2} \\
%Q^1_{2\bar 2} &= \tfrac{y(x|z|^2+y^3)}{(xy-|z|^2)^2} \\
%Q^1_{1\bar 2} &= \tfrac{z(x|z|^2+y^3)}{(xy-|z|^2)^2} \\
%Q^2_{1\bar 1} &= \tfrac{2y^2}{xy-|z|^2} \\
%Q^2_{2\bar 2} &= \tfrac{2|z|^2}{xy-|z|^2} \\
%Q^2_{1\bar 2} &= \tfrac{2iyz}{xy-|z|^2} \\
%Q^3_{1\bar 1} &= \tfrac{(x^2+y^2)|z|^2}{(xy-|z|^2)^2} \\
%Q^3_{2\bar 2} &= \tfrac{y^4+|z|^4}{(xy-|z|^2)^2} \\
%Q^3_{1\bar 2} &= \tfrac{z(x+iy)(|z|^2-iy^2)}{(xy-|z|^2)^2} \\
%Q^4_{1\bar 1} &= \tfrac{2y^2}{xy-|z|^2} \\
%Q^4_{2\bar 2} &= \tfrac{2|z|^2}{xy-|z|^2} \\
%Q^4_{1\bar 2} &= \tfrac{2iyz}{xy-|z|^2} \\
%\end{aligned}$$

\subsubsection*{Inoue surfaces of type $S^0$} \hfill \par
%$$
%\begin{aligned}
%S_{1\bar 1}&=\tfrac{x ((b^2+9a^2)|z|^2+4a^2(xy-|z|^2) )}{(xy-|z|^2)^2}\,\,,  &S_{2\bar 2}&= \tfrac{xy ((b^2+9a^2)|z|^2-8a^2(xy-|z|^2) )}{(xy-|z|^2)^2}\,\,, 	 &S_{1\bar 2}&= \tfrac{xz ((b^2+9a^2)xy -2a(a+ib)(xy-|z|^2) )}{(xy-|z|^2)^2}\,\, ,\\
%Q^1_{1\bar 1}&=\tfrac{x^2 ((b^2+9a^2)|z|^2+4a^2(xy-|z|^2) )}{(xy-|z|^2)^2}\,\, ,	 &Q^1_{2\bar 2}&=\tfrac{xy ((b^2+9a^2)|z|^2+4a^2(xy-|z|^2) )}{(xy-|z|^2)^2}\,\,,	 &Q^1_{1\bar 2}&=\tfrac{xz ((b^2+9a^2)|z|^2+4a^2(xy-|z|^2) )}{(xy-|z|^2)^2}\,\, ,\\
%Q^2_{1\bar 1}&= \tfrac{8a^2x^2}{xy-|z|^2}\,\, ,						 &Q^2_{2\bar 2}&=\tfrac{2(b^2+a^2)|z|^2}{xy-|z|^2}\,\,,		 &Q^2_{1\bar 2}&=\tfrac{-4a(a+ib)}{xy-|z|^2}\,\, ,\\
%Q^3_{1\bar 1}&=\tfrac{(b^2+9a^2)x^2|z|^2}{(xy-|z|^2)^2}\,\, ,	 &Q^3_{2\bar 2}&=\tfrac{(b^2+9a^2)|z|^4 +4a^2(xy+2|z|^2)(xy-|z|^2)}{(xy-|z|^2)^2}\,\,,	 &Q^3_{1\bar 2}&=\tfrac{xz ((b^2+9a^2)|z|^2 +2a(3a+ib)(xy-|z|^2) )}{(xy-|z|^2)^2}\,\, ,\\
%Q^4_{1\bar 1}&=\tfrac{8a^2x^2}{xy-|z|^2}\,\, ,						 &Q^4_{2\bar 2}&=\tfrac{2(b^2+a^2)|z|^2}{xy-|z|^2} \,\,,		 &Q^4_{1\bar 2}&=\tfrac{-4a(a+ib)}{xy-|z|^2}\,\, .
%\end{aligned}
%$$

$$\begin{aligned}
S_{1\bar 1} &= \tfrac{x ((b^2+9a^2)|z|^2+4a^2(xy-|z|^2) )}{(xy-|z|^2)^2} \,\,, \\
S_{2\bar 2} &= \tfrac{xy ((b^2+9a^2)|z|^2-8a^2(xy-|z|^2) )}{(xy-|z|^2)^2} \,\,, \\
S_{1\bar 2} &= \tfrac{xz ((b^2+9a^2)xy -2a(a+ib)(xy-|z|^2) )}{(xy-|z|^2)^2} \,\,, \\
Q^1_{1\bar 1} &= \tfrac{x^2 ((b^2+9a^2)|z|^2+4a^2(xy-|z|^2) )}{(xy-|z|^2)^2} \,\,, \\
Q^1_{2\bar 2} &= \tfrac{xy ((b^2+9a^2)|z|^2+4a^2(xy-|z|^2) )}{(xy-|z|^2)^2} \,\,, \\
Q^1_{1\bar 2} &= \tfrac{xz ((b^2+9a^2)|z|^2+4a^2(xy-|z|^2) )}{(xy-|z|^2)^2} \,\,, \\
Q^2_{1\bar 1} &= \tfrac{8a^2x^2}{xy-|z|^2} \,\,, \end{aligned}$$
$$\begin{aligned}
Q^2_{2\bar 2} &= \tfrac{2(b^2+a^2)|z|^2}{xy-|z|^2}\,\,,  \\
Q^2_{1\bar 2} &= \tfrac{-4a(a+ib)}{xy-|z|^2} \,\,, \\
Q^3_{1\bar 1} &= \tfrac{(b^2+9a^2)x^2|z|^2}{(xy-|z|^2)^2} \,\,, \\
Q^3_{2\bar 2} &= \tfrac{(b^2+9a^2)|z|^4 +4a^2(xy+2|z|^2)(xy-|z|^2)}{(xy-|z|^2)^2} \,\,, \\
Q^3_{1\bar 2} &= \tfrac{xz ((b^2+9a^2)|z|^2 +2a(3a+ib)(xy-|z|^2) )}{(xy-|z|^2)^2} \,\,, \\
Q^4_{1\bar 1} &= \tfrac{4a^2x^2}{xy-|z|^2} \,\,, \\
Q^4_{2\bar 2} &= \tfrac{(b^2+a^2)|z|^2}{xy-|z|^2}\,\,,  \\
Q^4_{1\bar 2} &= \tfrac{-2a(a+ib)}{xy-|z|^2} \,\,. \\
\end{aligned}$$
Here, $a,b\in\R$ denotes the parameters of the family of complex structures on Inoue surfaces of type $S^0$.

\subsubsection*{Inoue surfaces of type $S^{\pm}$} \hfill \par

For what concerns the complex structure $J_1$ on Inoue surfaces of type $S^\pm$, we get

$$\begin{aligned}
S_{1\bar 1} &= \tfrac{-2(xy-|z|^2)^2-xy(xy-|z|^2)-(z^2+\bar{z}^2)|z|^2+2xy|z|^2}{(xy-|z|^2)^2}\,\,, \\
S_{2\bar 2} &= \tfrac{y^2 (xy+|z|^2-(z^2+\bar{z}^2) )}{(xy-|z|^2)^2}\,\,, \\
S_{1\bar 2} &= \tfrac{xy^2(z-\bar{z})+yz(xy-z^2)}{(xy-|z|^2)^2}\,\,, \\
Q^1_{1\bar 1} &= \tfrac{xy (xy+|z|^2-(z^2+\bar{z}^2) )}{(xy-|z|^2)^2}\,\,, \\
Q^1_{2\bar 2} &= \tfrac{y^2 (xy+|z|^2-(z^2+\bar{z}^2) )}{(xy-|z|^2)^2} \,\,,\\
Q^1_{1\bar 2} &= \tfrac{yz ((z-\bar{z})|z|^2+xyz-z^3 )}{(xy-|z|^2)^2}\,\,, \\
Q^2_{1\bar 1} &= \tfrac{2|z|^2}{xy-|z|^2} \,\,,\\
Q^2_{2\bar 2} &= \tfrac{2y^2}{xy-|z|^2}\,\,, \\
Q^2_{1\bar 2} &= \tfrac{2y\bar{z}}{xy-|z|^2}\,\,, \\
Q^3_{1\bar 1} &= \tfrac{x^2y^2-xy(z^2+\bar{z}^2)+|z|^4}{(xy-|z|^2)^2} \,\,,\\
Q^3_{2\bar 2} &= \tfrac{y^2 (2|z|^2-(z^2+\bar{z}^2) )}{(xy-|z|^2)^2}\,\,, \\
Q^3_{1\bar 2} &= \tfrac{y(xy-z^2)(z-\bar{z})}{(xy-|z|^2)^2} \,\,,\\
Q^4_{1\bar 1} &= \tfrac{|z|^2}{xy-|z|^2} \,\,,\\
Q^4_{2\bar 2} &= \tfrac{y^2}{xy-|z|^2} \,\,,\\
Q^4_{1\bar 2} &= \tfrac{y\bar{z}}{xy-|z|^2}\,\,. \\
\end{aligned}$$

On the other hand, given the complex structure $J_2$ on Inoue surfaces of type $S^+$, we get

$$\begin{aligned}
S_{1\bar 1} &= \tfrac{xy^2(4x-y) -(2|z|^2 -2y^2 +z^2+\bar{z}^2)|z|^2 -y(7x-(z+\bar{z}))(xy-|z|^2)}{(xy-|z|^2)^2}\,\,, \\
S_{2\bar 2} &= \tfrac{y^2 (|z|^2+y^2+xy-(z^2+\bar{z}^2) )}{(xy-|z|^2)^2} \,\,,\\
S_{1\bar 2} &= \tfrac{y^2(xy-|z|^2)+xy^2(2z^2-\bar{z})+zy(y^2-z^2)}{(xy-|z|^2)^2}\,\,, \end{aligned}$$
$$\begin{aligned}
Q^1_{1\bar 1} &= \tfrac{xy (|z|^2-(z^2+\bar{z}^2)+y(x+y) )}{(xy-|z|^2)^2} \,\,,\\
Q^1_{2\bar 2} &= \tfrac{y^2 (|z|^2-(z^2+\bar{z}^2)+y(x+y) )}{(xy-|z|^2)^2} \,\,,\\
Q^1_{1\bar 2} &= \tfrac{y ((z-\bar{z})|z|^2 +xyz +y^2z -z^3 )}{(xy-|z|^2)^2} \,\,,\\
Q^2_{1\bar 1} &= \tfrac{2(|z|^2 +y(z+\bar{z}) +y^2 ))}{xy-|z|^2} \,\,,\\
Q^2_{2\bar 2} &= \tfrac{2y^2}{xy-|z|^2} \,\,,\\
Q^2_{1\bar 2} &= \tfrac{2y(y+\bar{z})}{xy-|z|^2}\,\,, \\
Q^3_{1\bar 1} &= \tfrac{2xy|z|^2-xy(z^2+\bar{z}^2)+y^2|z|^2 +(xy-y(z+\bar{z})-|z|^2)(xy-|z|^2)}{(xy-|z|^2)^2}\,\,, \\
Q^3_{2\bar 2} &= \tfrac{y^2 (2|z|^2-(z^2+\bar{z}^2)+y^2 )}{(xy-|z|^2)^2} \,\,,\\
Q^3_{1\bar 2} &= \tfrac{y (z|z|^2 +y|z|^2 -xy(z-\bar{z}) -z^3 +y^2z -xy^2 )}{(xy-|z|^2)^2}\,\,, \\
Q^4_{1\bar 1} &= \tfrac{|z|^2+y(z+\bar{z})+y^2}{xy-|z|^2} \,\,,\\
Q^4_{2\bar 2} &= \tfrac{y^2}{xy-|z|^2} \,\,,\\
Q^4_{1\bar 2} &= \tfrac{y(y+\bar{z})}{xy-|z|^2} \,\,.\\
\end{aligned}$$

%$$Q_{1\bar 1} =\frac{+xy^3-y^2|z|^2+(2|z|^2+2y(z+\bar{z})+y^2)(xy-|z|^2)}{2(xy-|z|^2)^2}$$

\bigskip\bigskip
\font\smallsmc = cmcsc8
\font\smalltt = cmtt8
\font\smallit = cmti8
\hbox{\parindent=0pt\parskip=0pt
\vbox{\baselineskip 9.5 pt \hsize=5truein
\obeylines
{\smallsmc
Dipartimento di Matematica e Informatica ``Ulisse Dini'', Universit$\scalefont{0.55}{\text{\Aac}}$ di Firenze
Viale Morgagni 67/A, 50134 Firenze, ITALY}
\smallskip
{\smallit E-mail adress}\/: {\smalltt francesco.pediconi@unifi.it
}
}
}

\bigskip\bigskip
\font\smallsmc = cmcsc8
\font\smalltt = cmtt8
\font\smallit = cmti8
\hbox{\parindent=0pt\parskip=0pt
\vbox{\baselineskip 9.5 pt \hsize=5truein
\obeylines
{\smallsmc
Dipartimento di Matematica ``Giuseppe Peano'', Universit$\scalefont{0.55}{\text{\Aac}}$ di Torino
Via Carlo Alberto 10, 10123 Torino, ITALY}
\smallskip
{\smallit E-mail adress}\/: {\smalltt mattia.pujia@unito.it
}
}
}

\end{document}